\documentclass{elsarticle}

\usepackage{latexsym,bm,amssymb,amsfonts,amsbsy,amsmath,graphics,graphicx,color}

\usepackage{multirow}%
\usepackage{amsthm}%
\usepackage{mathrsfs}%
\usepackage[title]{appendix}%
\usepackage{xcolor}%
\usepackage{textcomp}%
\usepackage{manyfoot}%
\usepackage{booktabs}%
\usepackage{algorithm}%
\usepackage{algorithmicx}%
\usepackage{algpseudocode}%
\usepackage{listings}%
\usepackage{comment}

\usepackage{geometry,graphicx,amssymb,amsfonts,amsmath,bm,url}
\usepackage{pifont}
\usepackage{url}
\usepackage{fancyhdr}
\usepackage{pstricks,pst-node,pst-text,pst-3d}

\newtheorem{assumption}{Assumption}
\newtheorem{proposition}{Proposition}
\newtheorem{lemma}{Lemma}

\newtheorem{theorem}{Theorem}

\theoremstyle{remark}
\newtheorem{remark}{Remark}

\def \RR {{\mathbb{R}}}

\def \x {{\bf x}}

\def \x {{\bf x}}

\DeclareMathOperator{\MSE}{{E}}
\DeclareMathOperator{\MSEuw}{{E}_{\mathrm{uw}}}
\DeclareMathOperator{\MSEmin}{{E}_{\mathrm{min}}}

\DeclareMathOperator{\diag}{diag}
\DeclareMathOperator*{\argmin}{arg\,min}

\newrgbcolor{marrone}{0.6156 0.3164 0.0313}
\newrgbcolor{DBlue}{0 0 0.7}
\newrgbcolor{ts}{0.50 0.598 0.781}
\newrgbcolor{verde}{0 0.4 0}
\newrgbcolor{LBlue}{0.58 0.65 0.9}
\newrgbcolor{Blue}{0.1 0.00 0.60}
\newrgbcolor{rosso}{0.9633    0.1641    0.0039}

\newrgbcolor{giorgia}{0.1 0.00 0.60}
\newrgbcolor{felice}{0.9633    0.1641    0.0039}
\newrgbcolor{pierluigi}{0.50 0.598 0.781}

\begin{document}
\begin{frontmatter}

\title{Maximum Entropy Least Squares Solutions \\ of Overdetermined Linear Systems}

\author[bari]{Felice~Iavernaro\corref{cor}}
\ead{felice.iavernaro@uniba.it }
\author[bari]{Monica Lazzo}
\ead{monica.lazzo@uniba.it }
\author[bari]{Lorenzo Pisani}
\ead{lorenzo.pisani@uniba.it }
\cortext[cor]{Corresponding author}

\address[bari]{Dipartimento di Matematica, Universit\`a degli Studi di Bari Aldo Moro, Italy}

\begin{abstract}
We investigate the theoretical foundations of a recently introduced entropy-based formulation of weighted least squares for the approximation of overdetermined linear systems, motivated by robust data fitting in the presence of sparse gross errors. 

The weight vector is interpreted as a discrete probability distribution and is determined by maximizing Shannon entropy under normalization and a prescribed mean squared error (MSE) constraint. 
Unlike classical ordinary least squares, where the error level is an output of the minimization process, here the MSE value plays the role of a control parameter, and entropy selects the least biased weight distribution achieving the prescribed accuracy.

The resulting optimization problem is nonconvex due to the nonlinear coupling between the weights and the solution induced by the residual constraint. 
We analyze the associated optimality system and characterize stationary points through first- and second-order conditions. 
We prove the existence and local uniqueness of a smooth branch of entropy-maximizing configurations emanating from the ordinary least squares solution and establish its global continuation under suitable nondegeneracy conditions. 
Furthermore, we investigate the asymptotic regime as the prescribed MSE tends to zero and show that, under appropriate assumptions, the limiting configuration concentrates on a largest subset of data consistent with the linear model, thus suppressing the influence of outliers.

Two numerical experiments illustrate the theoretical findings and confirm the robustness properties of the method.

\end{abstract}


\end{frontmatter}

\section{Introduction}
\label{sec_intro}

The approximation of overdetermined linear systems by least squares methods is a classical and fundamental topic in numerical linear algebra, optimization, and statistics. 
Given an overdetermined system $Ax=b$, the Ordinary Least Squares (OLS) solution minimizes the Euclidean norm of the residual vector and enjoys well-known optimality properties under standard modeling assumptions such as homoscedasticity and normality of errors (see, e.g., \cite{Montgomery2012,Seber2012}). 
From an optimization viewpoint, OLS corresponds to the unconstrained minimization of a strictly convex quadratic functional and therefore admits a unique global minimizer whenever $A$ has full column rank.

In many practical applications, however, the observation vector $b$ is contaminated not only by small Gaussian perturbations but also by sparse and possibly large anomalous deviations (outliers). 
In such situations, classical least squares may lose robustness, since even a few corrupted data points can significantly bias the solution and inflate the mean squared error (see \cite{Huber2009}). 
Weighted least squares offers additional modeling flexibility by assigning different influence to different observations, but it leaves open the delicate issue of determining the weights in a principled and stable way.

Recently, a maximum entropy principle has been proposed as a systematic criterion for selecting weights in least squares approximation. 
The key idea was introduced in \cite{GiIa21} and  consists in interpreting the weight vector as a discrete probability distribution and determining it by maximizing Shannon's entropy under the constraint that the associated weighted MSE equals a prescribed value $\MSE$. Unlike in OLS problems, where the error level is an output of the minimization process, here  $\MSE$ is an input parameter, and the entropy principle selects the least biased weight distribution achieving that accuracy level. This entropy-driven strategy has proved effective in enhancing robustness in spline approximation and signal processing applications, automatically reducing the influence of anomalous observations while preserving consistency with the bulk of the data \cite{BrGiIaRu24,DeFaIaLoMaRu25,AmBrIa26,AmFaGrIaLaMaNoRu}.

Although the entropy functional is strictly concave on the probability simplex, the MSE constraint induces a nonlinear coupling between the weights and the unknown solution, so that  the resulting optimization problem is nonconvex. 
Consequently, multiple local maximizers may co-exist, and the classical global convexity arguments are no longer applicable. 
This structural feature motivates a careful analytical investigation of local solution branches rather than global optimality in the entire feasible set.

The main contribution of this paper is a dynamical reinterpretation of the entropy-constrained problem. 
Starting from the first-order optimality conditions obtained via Lagrange multipliers, we reformulate the stationary equations as a nonlinear system depending on the parameter $\MSE$. 
Instead of treating this system purely as a family of static optimization problems, we regard $\MSE$ as a continuation parameter and analyze the set of stationary points as a solution manifold parameterized by $\MSE$.
More specifically, we show that in a neighborhood of the OLS configuration, corresponding to uniform weights, the stationary conditions define a locally unique smooth branch of solutions. 
This local result is obtained by applying the Implicit Function Theorem under suitable nondegeneracy assumptions. 
Subsequently, we transform the parameter-dependent nonlinear system into a differential system governing the evolution of the multipliers, weights, and solution vector with respect to  $\MSE$, which acts as the independent variable.   
In this way, the entropy-maximizing configurations are no longer viewed as isolated critical points, but as trajectories of a dynamical system in an extended state space. This dynamical formulation provides several advantages:
\begin{itemize}
\item it yields a constructive mechanism for continuing the solution branch beyond the local regime;
\item it clarifies the role of singularities as points where the Jacobian of the stationary system loses invertibility;
\item it allows a precise description of how the weight distribution progressively concentrates on subsets of data as $\MSE$ decreases.
\end{itemize}
From this perspective, we perform a global analysis by studying the maximal interval of existence of the associated Cauchy problem.
Under suitable nondegeneracy assumptions, we prove that the branch can be continued until either a degeneracy condition occurs or the limiting regime $\MSE \to 0$ is reached. 
In the latter case, we show that the entropy-maximizing strategy naturally isolates a largest subset of data consistent with the linear model, effectively suppressing outliers without requiring their a priori identification.

Therefore, the entropy-driven least squares method admits a unified interpretation: 
it generates a continuous path connecting the classical OLS solution to increasingly selective approximations obtained by progressively lowering the admissible error level. 
From this perspective, robustness emerges as a dynamical selection mechanism governed by the entropy principle.

It is worth emphasizing that  the present work has a primarily theoretical character. 
While the robustness properties and numerical performance of the entropy-driven weighted least squares methodology,  also in comparison with state-of-the-art robust techniques, have already been extensively documented in the aforementioned references, a comprehensive analytical foundation of the approach has so far been lacking. 
The main objective of this paper is precisely to fill this gap by providing a rigorous study of the structure of stationary points, their local and global continuation, and their asymptotic behavior. 
The simple numerical experiments included here serve exclusively to illustrate and validate the theoretical findings rather than to perform a new empirical assessment of the method.

The paper is organized as follows. 
In Section \ref{sec_formulation} we introduce the mathematical formulation of the maximum-entropy weighted least squares problem and derive the associated first-order optimality conditions. 
Section \ref{sec_local_results} is devoted to the local analysis near the OLS solution and to the proof of existence and uniqueness of a smooth solution branch. 
In Section \ref{sec_global_continuation} we develop the global continuation theory by reformulating the problem as a Cauchy problem parameterized by $\MSE$. 
Section~\ref{sec_asymptotic} presents an asymptotic analysis of the branch, including the limiting case $\MSE \to 0$. 
Section \ref{sec_numerics} provides numerical illustrations that confirm the theoretical results and highlight the actual behavior of the method. Finally, concluding remarks are presented in Section \ref{sec_conclusions}.

\section{Mathematical Formulation of the Maximum-Entropy Least Squares Problem}
\label{sec_formulation}

For given integers $m\ge n$,  matrix $A \in \mathbb{R}^{m \times n}$, vectors $b\in \mathbb{R}^m$ and $x\in \RR^n$,  the weighted least squares problem associated with the overdetermined, possibly inconsistent, linear system $Ax=b$ consists in finding 
\begin{equation}
\label{WLS-problem}
x(w) := \argmin_{x \in \mathbb{R}^n} \sum_{i=1}^m w_i (a_i^\top x - b_i)^2 = \argmin_{x \in \mathbb{R}^n} \|\sqrt{W}(Ax-b)\|_2^2,
\end{equation}
where $a_i^\top$ denotes the $i$th row of $A$, $w=(w_1,\dots,w_m)^\top \in \RR^m$ is a given  vector of positive weights and $W=\diag(w)$ denotes the diagonal matrix whose principal diagonal is $w$. Without loss of generality, the weights are assumed to satisfy the normalization condition 
\begin{equation}
\label{normalization}
\sum_{i=1}^n w_i = 1.
\end{equation} 
Assuming $A$ has full column rank, it is well known and easy to show that the solution of  problem  \eqref{WLS-problem} is unique and may be obtained, for example, as the solution of  the normal system
\begin{equation}
\label{normalsys}
A^\top W A x = A^\top W b.
\end{equation}
The corresponding  weighted Mean Squared Error (MSE) is 
\begin{equation}
\label{MSEbar}
E = \sum_{i=1}^m w_i (a_i^\top x(w) - b_i)^2.
\end{equation}
The use of the symbol $\MSE$ to denote the value of the MSE in \eqref{MSEbar}, in place of more standard notations,   is motivated by the fact that the  MSE will be repeatedly used as a scalar parameter in continuation arguments and as an independent variable in differential equations.

In practical real-world applications, the observation vector $ b \in \mathbb{R}^m $ may be decomposed in three distinct contributions:
$$
b = b_1 + \varepsilon + \delta,
$$
where each term reflects a different source of information or error. The first term, $ b_1 $, is assumed to perfectly conform to a theoretical linear model that accurately captures the ideal behavior of the underlying phenomenon under observation: $b_1 = A x_{\text{true}}$. The second term, $\varepsilon$, accounts for systematic errors induced by the finite precision of the measurement instruments such as sensors, analog-to-digital converters, or other electronic devices, and is modeled as a realization of a zero-mean Gaussian process with known and bounded variance: $\varepsilon \sim \mathcal{N}(0, \sigma^2)$. Finally, $ \delta $ represents sparse gross errors (outliers), arising from various unpredictable sources such as sensor failure, environmental interference, human error during data entry, or software glitches.

Ordinary Least Squares (OLS) corresponds to the classical choice $ w^* = \frac{1}{m} u $, with $ u = (1, \dots, 1)^\top $, representing a uniform distribution of weights. We denote by $ x^* $ the OLS solution,  $ r^* = A x^* - b $ the residual vector and $\MSEuw=\frac{1}{m}||r^*||_2^2$ the corresponding MSE ($\mathrm{uw}$ stands for {\em uniform weights}). This standard approximation strategy is appropriate when the third source of errors $\delta$ is absent or has small norm. In fact, it is well known that two  assumptions underlying the validity of the OLS approach are: \emph{homoscedasticity}, meaning that the residuals exhibit constant variance across the range of fitted values; \emph{normality of errors}, which implies that the residuals are approximately normally distributed \cite{Montgomery2012,Seber2012}. 

Violations of these assumptions can lead to inefficient estimates and unreliable inference. Typically, when comparing the associated MSE  in  absence ($\delta  =0$) and presence ($\delta \not =0$)  of  corrupted data, a substantial increase in the value of $\MSEuw$ can be observed: in such a case, non-uniform weight distributions could help in mitigating the impact of such violations \cite{Huber2009}.

A recently introduced robust approximation strategy, capable of effectively addressing the anomalies introduced by the third component $\delta$, adopts a maximum-entropy-based formulation of weighted least squares. In this approach, the normalization condition (\ref{normalization}) allows the weight vector to be interpreted as a discrete probability distribution, and the associated entropy
\begin{equation}
\label{entropy}
H(w) = -\sum_{i=1}^m w_i \log w_i
\end{equation}
is maximized under a prescribed MSE constraint \cite{GiIa21, BrGiIaRu24, DeFaIaLoMaRu25}. Specifically, we consider the family of constrained optimization problems parameterized by a target mean squared error level $\MSE \in (0, \MSEuw]$:

\begin{equation}
\label{MEWLS}
\begin{aligned}
\max_{w \in \mathbb{R}^m} \quad & -\sum_{i=1}^m w_i \log w_i \\
\text{subject to} \quad & \sum_{i=1}^m w_i = 1, \\
& w_i \geq 0 \quad \text{for all } i = 1, \dots, m, \\
& \sum_{i=1}^m w_i (a_i^\top x - b_i)^2 = \MSE.
\end{aligned}
\end{equation}

Unlike in ordinary least squares problems, where the mean squared error $\MSEuw$ emerges as an output of the minimization process, in this entropy-based formulation $\MSE$ is an input parameter: the maximum entropy principle is employed to guide the selection of the weight distribution that achieves the specified MSE level.
\begin{remark}
\label{remH}
{\em
When $\MSE = \MSEuw$, the unique {\em global maximizer} of \eqref{MEWLS} is the uniform weight vector $w^* = \frac{1}{m} u$, where $u = (1, \dots, 1)^\top$. However, for any $\MSE < \MSEuw$, $w^*$ becomes infeasible since the weights are to be reallocated in order to obey the MSE constraint. Due to the strict concavity of the entropy function $H(w)$ in the probability symplex, any feasible $w \ne w^*$ satisfies $H(w) < H(w^*)=\log m$. Hence, entropy necessarily decreases as the target $\MSE$ drops below $\MSEuw$.
}
\end{remark}

The present study aims to investigate the existence, uniqueness and the behavior of {\em local maximizers} of \eqref{MEWLS} as $\MSE$ decreases from $\MSEuw$. In particular, we seek to determine whether a smooth branch of local solutions $(x(\MSE), w(\MSE))$ exists for $\MSE < \MSEuw$, and to characterize the  properties of such a branch.

It is worth noting that while the entropy function $w \mapsto -\sum_i w_i \log w_i$ is strictly concave on the probability simplex, the MSE constraint
$$
f(x, w) := \sum_{i=1}^m w_i (a_i^\top x - b_i)^2 = \MSE
$$
defines a generally nonconvex set in the $(x,w)$ space. 
%The analysis of the feasible set geometry is thus crucial for understanding the structure of the solution manifold.
To see this, we analyze the structure of the Hessian matrix of the function $f(x,w)$. Define the residuals as $r_i = a_i^\top x - b_i$, and let $r = Ax - b \in \mathbb{R}^m$. The Hessian matrix of $f$ with respect to the joint variable $(x, w)$ is a block matrix given by
\begin{equation}
\label{Hf}
\nabla^2_{(x,w)} f(x, w) =
2 \begin{bmatrix}
A^\top W A & A^\top\diag(r) \\
\diag(r) A & 0
\end{bmatrix}.
\end{equation}

Under the assumption that $A$ has full column rank, the matrix $A^\top W A$ is symmetric positive definite. Applying Schur's criterion, $\nabla^2_{(x,w)} f(x, w)$ is positive semidefinite if and only if the Schur complement 
$$
S = - \diag(r) A  (A^\top W A)^{-1} A^\top\diag(r)
$$
is positive definite. However, this matrix is always negative semidefinite, and strictly negative definite for all directions $z$ such that  $\diag(r)z \not \in \mathrm{ker}(A^\top)$. This implies that the full Hessian (\ref{Hf}) is indefinite. Therefore, the function $f(x, w)$ is not jointly convex in $(x, w)$, and its level sets
$$
\{ (x, w) \in \mathbb{R}^d \times \mathbb{R}^m : f(x, w) = \MSE \}
$$
are, in general, not convex sets. Consequently, problem \eqref{MEWLS} is nonconvex, and multiple local maxima may exist for $\MSE < \MSEuw$. This justifies our focus on local, rather than global optimality when looking at the branch of solutions originating from the OLS configuration. 

Furthermore, since to the OLS solution there corresponds the uniform weight distribution, the branch will lie in the interior of the feasible compact set defined by the constraints so that, to solve problem \eqref{MEWLS}, we apply the classical method of Lagrange multipliers, introducing multipliers only for the {\em equality} constraints. Any stationary point of the Lagrangian necessarily satisfies $w_i>0$, so the inequality constraints are implicitly enforced  by the domain of the objective function. In fact, considering the Lagrangian function 
\begin{equation}
\label{Lagrangian}
\begin{array}{rl}
\mathcal{L}(\lambda, \mu, w, x) = & \displaystyle
\sum_{i=1}^{m} w_i \log w_i 
+ \lambda \left( \sum_{i=1}^{m} w_i -1 \right) + \mu \left( \sum_{i=1}^{m} w_i (a_i^\top x - b_i)^2 - \MSE  \right),
\end{array}
\end{equation}
its gradient with respect to $w$ reads
\begin{equation}
\label{Dwk}
\frac{\partial \mathcal{L}}{\partial w} = \log w + (1+\lambda)u +\mu (Ax-b)^2,
\end{equation}
where, for a given vector $z$, the notation $z^2$ denotes the Hadamard (element-wise) square, i.e., $z^2 = z \odot z$. Setting \eqref{Dwk} to zero and solving for $w$ gives
\begin{equation}
\label{w-expr}
w=\exp(-(1+\lambda))\exp\left(-\mu (Ax-b)^2\right),
\end{equation}
which ensures that $w_i > 0$ for all $i$, at any stationary point of $\mathcal{L}$. Differentiating $\mathcal{L}$ with respect to $\lambda$, $\mu$, and $x$, we obtain
\begin{equation}
\label{Dlam}
\frac{\partial \mathcal{L}}{\partial \lambda}=\sum_{i=1}^{m} w_i - 1,
\end{equation}
\begin{equation}
\label{Dmu}
\frac{\partial \mathcal{L}}{\partial \mu} = \sum_{i=1}^{m} w_i (a_i^\top x - b_i)^2 - \MSE,
\end{equation}
and
\begin{equation}
\label{Dx}
\frac{\partial \mathcal{L}}{\partial x} = 2 \mu (A^\top W A x - A^\top W b).
\end{equation}
When seeking the stationary points of $\mathcal{L}$, the contribution of this latter component splits in two different 
equations:
\begin{itemize}
\item $\mu=0$, which deactivates the MSE constraints and leads back to the OLS approximation
\begin{equation}
\label{critical0}
(\lambda^*,\mu^*,(w^*)^\top,(x^*)^\top)^\top=\left(-(1+\log \frac{1}{m}),0,\frac{1}{m}u^\top,({(A^\top A)}^{-1}{A^\top b})^\top\right)^\top,
\end{equation}
\item  $A^\top W A x = A^\top W b $, that is the normal system (\ref{normalsys}). 
\end{itemize}
Therefore, to compute the points along the solution branch, we study the zero set of the nonlinear vector function
\begin{equation}
\label{sys10}
F(\lambda,\mu,w,x) = (F_1, F_2, F_3^\top, F_4^\top)^\top := \left(\frac{\partial {\cal L}}{\partial \lambda}, \frac{\partial {\cal L}}{\partial \mu}, \frac{\partial {\cal L}}{\partial w}^\top, \frac{1}{2\mu}\frac{\partial {\cal L}}{\partial x}^\top \right)^\top.
\end{equation} 

In the next two sections we carry out a local analysis, where $ \MSE $ is close to $ \MSEuw $, and a global analysis, where $\MSE$ is reduced in order to make the impact of the anomaly term $\delta$ on the model essentially negligible.

In the particular case where $\MSE=0$, we show that, under suitable assumptions, the entropy-maximizing strategy naturally isolates the largest consistent subset of the data, suppressing outliers without requiring their a priori  identification.

\section{Local results}
\label{sec_local_results}
In this section, we investigate the solvability of problem (\ref{WLS-problem}) for values of the parameter $\MSE$ lying in a left neighborhood of $\MSEuw$. Rather than applying the implicit function theorem directly to the stationarity conditions $\nabla \mathcal{L} = 0$, whose Jacobian is singular at the critical point $(\lambda^*,\mu^*,w^*,x^*)^\top$, we consider instead the nonlinear system 
\begin{equation}
\label{sys1}
F(\lambda,\mu,w,{x}) = (F_1, F_2, F_3^\top, F_4^\top)^\top = 0,
\end{equation} 
with $F$ defined in \eqref{sys10}. This reformulation allows for a regular Jacobian and enables the application of the implicit function theorem. In a second step, we analyze the bordered Hessian of $\mathcal{L}$ to verify that the obtained stationary points indeed correspond to local solutions of problem (\ref{WLS-problem}).

\begin{theorem}
\label{theo1}
Assume that the vector $|r^*| = |Ax^*-b|$ is not constant. Then, there exist $\Delta > 0$ and unique functions $\lambda(\MSE)$, $\mu(\MSE)$, $w(\MSE)$, and $x(\MSE)$, defined on the interval $(\MSEuw - \Delta, \MSEuw)$, such that:
\begin{itemize}
    \item[(a)] $F\left(\lambda(\MSE), \mu(\MSE), w(\MSE), x(\MSE)\right) = 0$;
    \item[(b)] $\mu(\MSE) > 0$;
    \item[(c)] the critical points $\left(\lambda(\MSE), \mu(\MSE), w(\MSE)^\top, x(\MSE)^\top\right)^\top$ solve the constrained entropy maximization problem (\ref{WLS-problem}).
\end{itemize}
\end{theorem}
\begin{proof}
In compact notation, the four components of $F$ are defined as
\begin{subequations}
\label{sys11}
\begin{align}
F_1&:= u^\top w - 1, \label{sys1a} \\[.3cm]
F_2&:= w^\top (Ax-b)^2-\MSE, \label{sys1b} \\[.3cm]
F_3&:=\log w +  (1+\lambda)u +\mu (Ax-b)^2, \label{sys1c} \\[.3cm]
F_4&:= A^\top W A x - A^\top W b. \label{sys1d} 
\end{align}
\end{subequations}
Its Jacobian is given by
\begin{equation}
\label{JF}
J(\lambda,\mu,w,x) =
\begin{pmatrix}
0 & 0               & u^\top                 & 0^\top \\
0 & 0               & {(Ax-b)^2}^\top & 2 \left(A^\top W (Ax-b) \right)^\top \\
u & (Ax-b)^2 & W^{-1}          & 2\mu \diag(Ax-b)A \\
0 & 0               & \left(\diag(Ax-b)A \right)^\top & A^\top W A
\end{pmatrix}.
\end{equation}
Evaluating this at the critical point $(\lambda^*, \mu^*, w^*, x^*)$ yields the block-triangular matrix
$$
J^* =
\begin{pmatrix}
0 & 0 & u^\top & 0^\top \\
0 & 0 & (r^*)^2{}^\top & 0 \\
u & (r^*)^2 & m I & 0 \\
0 & 0 & \left(\diag(r^*)A \right)^\top & \frac{1}{m}A^\top A
\end{pmatrix},
$$
because $\mu^*=0$, $W=\diag(w^*)=\frac{1}{m}I$, with $I$ the identity matrix of dimension $m$, and $A^\top W (Ax-b)$ vanishes at any solutions of  (\ref{sys1d}).

Since  $A$ has full column  rank, $A^\top  A$ is positive definite. Furthermore, the assumption on $r^*$ implies that $u$ and $(r^*)^2$ are linearly independent so the block $[u,\,  (r^*)^2]$ has full rank. Therefore, the $3\times 3$ leading  block-submatrix (in the upper-left corner) is non singular and so is the full Jacobian matrix $J^*$. 

By the implicit function theorem, it follows that there exists a neighborhood of $\MSEuw$ in which the solution $(\lambda, \mu, w, x)$ of (\ref{sys1}) exists, is unique, and  depends smoothly on $\MSE$.

\smallskip

We now analyze the behavior of the Lagrange multiplier $\mu=\mu(\MSE)$ in order to prove property~(b). For brevity, we omit the dependence on $\MSE$ in $x_i(\MSE)$, $w_i(\MSE)$ and $r_i(\MSE)$ throughout the following computations. From \eqref{sys1c}, any critical point of $F$ satisfies
\begin{equation}
\label{eq:stationarity_log}
-\log(w)= \mu (Ax-b)^2 + (1+\lambda)u.
\end{equation}
Since $H(w)=-\sum_i w_i\log w_i$, differentiation along the solution branch gives
\begin{equation}
\label{eq:H_derivative_2}
\frac{d}{d\MSE} H(w(\MSE))
=
-\sum_i (\log w_i +1)\, w_i'.
\end{equation}
Using~\eqref{eq:stationarity_log}, we rewrite
\[
\log w_i +1 = -\mu r_i^2 - \lambda,
\qquad r_i^2 := (Ax-b)_i^2,
\]
and therefore
\begin{equation}
\label{eq:H_derivative_3}
\frac{d}{d\MSE} H(w(\MSE))
=
\mu \sum_i r_i^2 w_i'
+
\lambda \sum_i w_i'.
\end{equation}
Differentiating the normalization constraint $w^\top u=1$ yields
\[
\sum_i w_i' = 0.
\]
Differentiating the MSE constraint $w^\top r^2=\MSE$ along the solution branch yields
\[
\sum_i w_i' r_i^2
+
\sum_i w_i (r_i^2)' = 1.
\]
Since $(r_i^2)' = 2 r_i r'_i = 2 r_i (A x')_i$, we obtain
\[
\sum_i w_i (r_i^2)'
=
2 (Ax-b)^\top W A x',
\qquad W:=\mathrm{diag}(w).
\]
The stationarity condition with respect to $x$, namely
$A^\top W (Ax-b)=0$,
implies $(Ax-b)^\top W A x' = 0$. Therefore
\[
\sum_i w_i (r_i^2)' = 0
\]
and consequently
\[
\sum_i r_i^2 w_i' = 1.
\]
Substituting into~\eqref{eq:H_derivative_3}, we obtain
\begin{equation}
\label{eq:H_derivative_final}
\frac{d}{d\MSE} H(w(\MSE)) = \mu.
\end{equation}

This identity shows that the Lagrange multiplier $\mu$ coincides with the sensitivity of the optimal entropy value with respect to the exogenous variable $\MSE$, in agreement with the envelope principle for equality-constrained optimization problems. We had already observed in Remark \ref{remH} that, as $\MSE$ decreases from $\MSEuw$, the entropy also decreases, implying $\mu(\MSE) > 0$ for $\MSE \in (\MSEuw - \Delta, \MSEuw)$, provided $\Delta$ is sufficiently small. 

\smallskip

Finally, to show property (c), we  consider the bordered Hessian $\mathcal{H}$ evaluated at a critical point  $(\lambda(\MSE),\mu(\MSE),w(\MSE),x(\MSE))$ obtained above. Observe that, taking into account the Lagrangian function definition in  (\ref{Lagrangian}), we have to show that the critical point is indeed a minimum for $-H(w)$.  The bordered Hessian of $H$ takes the form
\begin{equation}
\label{borderedH}
\mathcal{H} =
\begin{pmatrix}
0 & 0               & u^\top                 & 0^\top \\
0 & 0               & {(Ax-b)^2}^\top &  0 \\
u & (Ax-b)^2 & W^{-1}          & 2\mu \diag(Ax-b)A \\
0 & 0               & 2\mu \left(\diag(Ax-b)A \right)^\top & 2\mu A^\top W A
\end{pmatrix}
\end{equation}
where, for notational convenience, the explicit dependence of all variables on $\MSE$ has been suppressed. To conclude that the critical point corresponds to a local minimum of $- H$, it suffices to show that  the bottom-right block 
\begin{equation}
\label{Htmatrix}
\widetilde {\cal H} =
\begin{pmatrix}
 W^{-1}          & 2\mu \diag(Ax-b)A \\
2\mu \left(\diag(Ax-b)A \right)^\top & 2\mu A^\top W A
\end{pmatrix}
\end{equation}
is positive definite on the subspace orthogonal to the (linearly independent) gradients of the two constraints \cite{BlBr92,LuYe08}. We show the stronger  condition that  $\widetilde{\cal H}$ is itself positive definite for $\Delta>0$ small enough.

Since $W^{-1}$ is positive definite, the matrix $\widetilde {\cal H}$ will be positive definite if and only if the Schur complement 
\begin{equation}
\label{Smatrix}
S(\MSE)= 2\mu A^\top W A - 4\mu^2 A^\top \diag(Ax-b) W \diag(Ax-b)A
\end{equation}
is itself positive definite. In the given interval $(\MSEuw-\Delta,\MSEuw)$,  we introduce the approximations
$$
\mu(\MSE)=O(\Delta),\quad  W=\frac{1}{m} I +O(\Delta), \quad  Ax-b = r^*+O(\Delta)
$$
and deduce the following expansion of $S(\MSE)$:
$$
S(\MSE)= \frac{2\mu}{m} A^\top \left( I -  2\mu  (\diag(r^*))^2 \right) A +O(\Delta^3) = \frac{2\mu}{m} A^\top A +O(\Delta^2)
$$
which is positive definite for $\Delta$ sufficiently small due to the full column rank of $A$ and the positivity of $\mu(\MSE)$.  Thus, we conclude that the constrained entropy maximization problem admits a smooth local branch of solutions in a left neighborhood of $\MSEuw$.
\end{proof}

\section{Global continuation of the solution branch}
\label{sec_global_continuation} 
Theorem \ref{theo1}, the solution branch
$
\MSE \longmapsto (w(\MSE),x(\MSE),\mu(\MSE),\lambda(\MSE))
$
emanating from the uniform--weights configuration at $\MSE=\MSEuw$ is well defined and smooth in a left neighborhood of $\MSEuw$. We therefore consider an interval of the form
$(\MSE_{\min},\MSEuw]$ on which the branch exists, where $\MSE_{\min}\ge 0$ denotes the leftmost endpoint of existence. In the remainder of this paper, we implicitly assume that the solution branch exists for the values of $\MSE$ under consideration.
In particular, whenever asymptotic properties as $\MSE\to 0^+$ are discussed, it is understood that the branch can be continued down to arbitrarily small positive values of $\MSE$.
The purpose of the subsequent analysis is precisely to justify this continuation by ruling out possible obstructions.
Our global analysis proceeds in four steps:

\begin{enumerate}
  \item We rewrite the stationarity conditions as an autonomous initial value problem (IVP) of the form $y' = f(y)$.
  \item We study  the invertibility of the Jacobian $J$ needed to define $f$ and  the regularity of the vector field $f$. 
  \item We illustrate the possible outcomes for global continuation, namely finite breakdown or extension to $\MSE\!\downarrow\!0$, making use of standard continuation results for initial value problems.  
  \item Finally, in Section \ref{sec_asymptotic}, we provide an asymptotic analysis of the the branch as $\MSE\to 0^+$.
\end{enumerate}

Recall the nonlinear mapping $F(\lambda,\mu,w,x)$ defined in \eqref{sys11}
and its Jacobian $J(\lambda,\mu,w,x)=\nabla_{(\lambda,\mu,w,x)}F$  displayed in \eqref{JF}.

\medskip
%%%%%%%%%%%%%%%%
\subsection{IVP formulation}
\label{subsec_IVP}
%%%%%%%%%%%%%%%%
Let $y = (\lambda,\mu,w^\top,x^\top)^\top$ and consider the stationarity system
\begin{equation}
\label{Fsys1}
F(y(\MSE);\MSE) = 0,
\end{equation}
where we have emphasized the  dependence of $F$ on the parameter $\MSE$ through (\ref{sys1b}).
Differentiating the identity $F(y(\MSE);\MSE)\equiv 0$ (with  $y(\MSE)$ a solution of \eqref{Fsys1}) with respect
to the scalar parameter $\MSE$ yields
\begin{equation}
\label{diff_identity}
\nabla_y F(y(\MSE);\MSE)\, y'(\MSE) + \partial_{\MSE} F(y(\MSE);\MSE)=0.
\end{equation}
Using the explicit form of $F$ (see \eqref{sys11}), one readily checks 
\[
\partial_{\MSE} F = -(0,1,0^\top,0^\top)^\top =: -e_2.
\]
Whenever the Jacobian $J(y(\MSE)) := \nabla_y F(y(\MSE);\MSE)$ is
invertible, we may solve for $y'(\MSE)$ and obtain the final value problem
\begin{equation}
\label{eq:IVP}\left\{
\begin{array}{ll}
y'(\MSE) = f(y(\MSE)) := J(y(\MSE))^{-1} e_2, & \MSE \le \MSEuw,\\[.15cm]
y(\MSEuw) = y^*,
\end{array}\right.
\end{equation}
where $y^*$ is the OLS critical point \eqref{critical0} (initial data). Note that by introducing the reparametrization  $\MSE \rightarrow -\MSE$, we can recast \eqref{eq:IVP} as an initial value problem (IVP).
Since the literature is primarily concerned with IVPs, we will, up to this reparametrization, refer to \eqref{eq:IVP} as an IVP.

Thus, the constrained stationary problem is locally equivalent to the
autonomous  differential equation $y' = f(y)$ whose state vector evolves on the manifold
of feasible stationary points (whenever the Jacobian is invertible). If $J(y)$ is invertible and the function $f(y)$ is locally Lipschitz, then the Cauchy–Lipschitz theorem ensures local existence and uniqueness of the branch, which has been addressed in the previous section by exploiting the implicit function theorem.

We note that,  $y(\MSE)$ being a solution of \eqref{Fsys1}, we have $A^\top W (Ax(\MSE)-b)=0$, so  the structure of the Jacobian simplifies as (compare with \eqref{JF}):
\begin{equation}
\label{J}
J(y) =
\begin{pmatrix}
0 & 0               & u^\top                 & 0^\top \\
0 & 0               & r^2{}^\top             & 0^\top \\
u & r^2 & W^{-1}          & 2\mu\,\diag(r)A \\
0 & 0               & (\diag(r)A)^\top & A^\top W A
\end{pmatrix},
\end{equation}
where  $r=Ax-b$, $r^2 := (r_1^2,\dots,r_m^2)^\top$, $W=\diag(w)$.

\begin{remark}
\label{remIVP} \em
Below in Proposition \ref{prop:mu_monotone} we show the positiveness of $\mu(\MSE)$ along the branch. From this result, it follows that along any solution of \eqref{eq:IVP}, for every $\MSE<\MSEuw$, the diagonal matrix 
$$
D_\mu:=\begin{pmatrix} 
I_{2+m} \\ & 2\mu(\MSE) I_n 
\end{pmatrix},
$$ 
is nonsingular. Since $D_\mu e_2=e_2$ and $D_\mu J(y) = \mathcal{H}$,  the bordered Hessian matrix \eqref{borderedH} is invertible whenever $J(y)$ is. Consequently, setting for example $y_0=y(\MSEuw-\delta)$ with  $0<\delta<\Delta$ as in Theorem \ref{theo1}, the IVP~\eqref{eq:IVP} admits the same solution as the system 
\begin{equation}
\label{eq:IVPH}\left\{
\begin{array}{ll}
z'(\MSE) = \mathcal{H}^{-1}e_2, & \MSE \le \MSEuw-\delta,\\[.15cm]
z(\MSEuw-\delta) = y_0,
\end{array}\right.
\end{equation}
This alternative formulation has the additional advantage of allowing us to incorporate, in the subsequent analysis, sufficient conditions ensuring the positive definiteness of the matrix $\widetilde {\cal H}$, defined at \eqref{Htmatrix},  on the subspace orthogonal to the  gradients of the two constraints, as required by the second–order optimality condition, thus guaranteeing that, not only the Cauchy problem does admit a solution, but also that the stationary point corresponds to a maximum of the entropy function.
\end{remark}

\medskip
%%%%%%%%%%%%%%%%
\subsection{Invertibility of $J$ and smoothness of $f$}
\label{subsec_regularity_Jacobian}
%%%%%%%%%%%%%%%%
We first show that, indeed, $\mu(\MSE)$ is positive and strictly increasing as $\MSE$ decreases. 

\begin{lemma}
\label{lem:partition}
For $i=1,\dots,m$, the weights $w_i(\mu)$ are implicitly determined by the following Gibbs equations 
\begin{equation}
\label{w-mu}
w_i(\mu)=\frac{e^{-\mu r_i(\mu)^2}}{Z(\mu)},
\qquad \text{with }
Z(\mu)=\sum_{j=1}^m e^{-\mu r_j(\mu)^2}.
\end{equation}
\end{lemma}
\begin{proof}
Relations \eqref{w-mu} come from \eqref{w-expr} after removing the unknown $\lambda$. In fact, multiplying both sides of \eqref{w-expr} by $u^\top$ and exploiting the normalization condition on the weights \eqref{normalization} yield
$$
\exp(1+\lambda)=u^\top \exp(-\mu r(\mu)^2) = Z(\mu).
$$
Substituting this expression for $\exp(1+\lambda)$ back into \eqref{w-expr} leads to \eqref{w-mu}.
\end{proof}

\begin{proposition}[Monotonicity and positivity of $\mu(E)$]
\label{prop:mu_monotone}
Let $A\in\mathbb R^{m\times n}$ have full column rank and let
$r(x)=Ax-b$.
For each $E\in(0,E_{\mathrm{uw}}]$, consider the entropy maximization problem
\begin{equation}
\label{eq:value_problem}
\mathcal V(E)
=
\sup_{x\in\mathbb R^n,\; w\in(0,\infty)^m}
\left\{
H(w)
\;\middle|\;
\sum_{i=1}^m w_i r_i(x)^2 = E,
\quad
\sum_{i=1}^m w_i = 1
\right\},
\end{equation}
where $H(w)=-\sum_{i=1}^m w_i\log w_i$.
Assume that, for every $E\in(\MSEmin,\MSEuw]$, there exists a smooth branch
\[
E \longmapsto (x(E),w(E),\mu(E),\lambda(E))
\]
satisfying the Karush--Kuhn--Tucker conditions associated with
\eqref{eq:value_problem}.
Then the following statements hold:
\begin{enumerate}
\item the value function $\mathcal V(E)$ is strictly concave;
\item $\mu(E)>0$ for all $E\in(0,E_{\mathrm{uw}})$;
\item $\mu(E)$ is strictly decreasing in $E$; equivalently,
$\mu(E)$ is strictly increasing as $E$ decreases;
\item if the branch extends down to $E\to0^+$, then $\mu(E)\to+\infty$.
\end{enumerate}
\end{proposition}

\begin{proof}
For fixed $E>0$, consider the Lagrangian\footnote{This Lagrangian differs from \eqref{Lagrangian} by an overall sign, reflecting the change from  a minimization to a maximization formulation. Since it is only used within the present proof, we retain the same notation.}
  
\[
\mathcal L(w,x,\mu,\lambda;E)
=
H(w)
+
\mu\!\left(E-\sum_{i=1}^m w_i r_i(x)^2\right)
+
\lambda\!\left(1-\sum_{i=1}^m w_i\right),
\]
with multipliers $\mu,\lambda\in\mathbb R$.
Define the associated dual function
\[
g(\mu,\lambda;E)
:=
\sup_{x\in\mathbb R^n,\; w\in(0,\infty)^m}
\mathcal L(w,x,\mu,\lambda;E).
\]
For every $(\mu,\lambda)$ and every feasible $(x,w)$ satisfying the constraints in \eqref{eq:value_problem}, one has
$\mathcal L(w,x,\mu,\lambda;E)=H(w)$, and therefore
\[
\mathcal V(E)\le g(\mu,\lambda;E).
\]
Taking the infimum over all multipliers yields
\begin{equation}
\label{eq:weak_duality}
\mathcal V(E)\le \inf_{\mu,\lambda} g(\mu,\lambda;E).
\end{equation}

By assumption, for each $E$ there exists a quadruple
$(x(E),w(E),\mu(E),\lambda(E))$ satisfying the KKT conditions.
In particular, $(x(E),w(E))$ maximizes $\mathcal L(\cdot,\cdot,\mu(E),\lambda(E);E)$,
 hence
\[
g(\mu(E),\lambda(E);E)
=
\mathcal L(w(E),x(E),\mu(E),\lambda(E);E)
=
H(w(E))
=
\mathcal V(E).
\]
Combining this identity with \eqref{eq:weak_duality}, we obtain 
\begin{equation}
\label{eq:value_dual_representation}
\mathcal V(E)=\inf_{\mu,\lambda} g(\mu,\lambda;E).
\end{equation}
Now, for every fixed pair $(\mu,\lambda)$, the function $g(\mu,\lambda;E)$ depends
affinely on $E$, since
\[
g(\mu,\lambda;E)
=
\mu E
+
\sup_{x,w}
\Bigl(
H(w)
-
\mu\sum_{i=1}^m w_i r_i(x)^2
+
\lambda(1-\sum_{i=1}^m w_i)
\Bigr).
\]
Being the infimum of a family of affine functions, $\mathcal V(E)$ is therefore
concave on $(0,E_{\mathrm{uw}}]$. In fact, let $E_1,E_2\in(0,E_{\mathrm{uw}}]$ and let $\theta\in[0,1]$.
Set
\[
E_\theta := \theta E_1 + (1-\theta)E_2.
\]
For every fixed pair $(\mu,\lambda)$, the dual function has the form
\[
g(\mu,\lambda;E)
=
\mu E + c(\mu,\lambda),
\]
where
\[
c(\mu,\lambda)
=
\sup_{x\in\mathbb R^n,\; w\in(0,\infty)^m}
\Bigl(
H(w)
-
\mu\sum_{i=1}^m w_i r_i(x)^2
+
\lambda(1-\sum_{i=1}^m w_i)
\Bigr)
\]
does not depend on $E$.
Hence $g(\mu,\lambda;E)$ is affine in $E$,  therefore
\[
g(\mu,\lambda;E_\theta)
=
\theta g(\mu,\lambda;E_1)
+
(1-\theta) g(\mu,\lambda;E_2)
\qquad
\forall\,(\mu,\lambda).
\]
Taking the infimum with respect to $(\mu,\lambda)$ on both sides yields
\[
\mathcal V(E_\theta)
=
\inf_{\mu,\lambda} g(\mu,\lambda;E_\theta)
=
\inf_{\mu,\lambda}
\Bigl(
\theta g(\mu,\lambda;E_1)
+
(1-\theta) g(\mu,\lambda;E_2)
\Bigr).
\]
Using the elementary inequality
\[
\inf_{z} \bigl(\theta a_z + (1-\theta)b_z\bigr)
\;\ge\;
\theta \inf_{z} a_z + (1-\theta)\inf_{z} b_z,
\]
valid for any real families $\{a_z\},\{b_z\}$ and $\theta\in[0,1]$, we obtain
\[
\mathcal V(E_\theta)
\ge
\theta \inf_{\mu,\lambda} g(\mu,\lambda;E_1)
+
(1-\theta)\inf_{\mu,\lambda} g(\mu,\lambda;E_2).
\]
Recalling the definition of $\mathcal V$, this implies
\[
\mathcal V(\theta E_1 + (1-\theta)E_2)
\ge
\theta \mathcal V(E_1) + (1-\theta)\mathcal V(E_2),
\]
which proves that $\mathcal V$ is concave.

We now prove strict concavity of $\mathcal V$. As established previously (see \eqref{eq:H_derivative_final}), along the smooth
KKT branch one has the envelope identity
\[
\mathcal V'(E)=\mu(E)
\]
Assume, by contradiction, that $\mu(E)=\mathcal V'(E)=\bar\mu$ for all $E\in[E_1,E_2]$, so $\mathcal V$ is concave but not strictly concave on  $[E_1,E_2]$. From Lemma \ref{lem:partition}, for each such $E$, any  optimal triple $(w(E),x(E),\mu(E))$
satisfies  the Gibbs equations 
\[
w_i(E)
=
\frac{\exp\bigl(-\bar\mu\, r_i(x(E))^2\bigr)}
{\sum_{j=1}^m \exp\bigl(-\bar\mu\, r_j(x(E))^2\bigr)},
\]
together with the normal equations
\[
A^\top W(E)A\,x(E)=A^\top W(E)b.
\]
These equations form a coupled system that depends only on $\bar\mu$ and
does not involve the parameter $E$ explicitly.
Under the standing assumptions, this system admits at most one solution
$(\bar w,\bar x)$.
Hence $(w(E),x(E))=(\bar w,\bar x)$ for all $E\in[E_1,E_2]$. Since the constraint is active along the branch, one has
\[
E = w(E)^\top r(x(E))^2.
\]
Using $(w(E),x(E))=(\bar w,\bar x)$ for all $E\in[E_1,E_2]$, we obtain
\[
E
=
\bar w^\top r(\bar x)^2
=: \bar E,
\qquad
\forall E\in[E_1,E_2],
\]
which contradicts the assumption $E_1<E_2$. Therefore $\mu(E)$ cannot be constant on any nontrivial interval,
and consequently $\mathcal V$ is strictly concave.

Since $\mathcal V$ is strictly concave, its derivative $\mathcal V'(E)=\mu(E)$ is
monotone strictly decreasing in $E$. Moreover, $\mu(E_{\mathrm{uw}})=0$ corresponds to the unconstrained (uniform--weight) solution, hence $\mu(E)>0$ and $\mu'(E)<0$ for all $E\in(E_{\mathrm{min}},E_{\mathrm{uw}})$.

\smallskip  If $E_{\mathrm{min}}=0$ and  $E\to0^+$ along the branch, then at least one residual component
$r_k(x(E))^2$ remains bounded away from zero (unless an exact interpolation
occurs, which is excluded). Since
\[
w_k(E)=\frac{e^{-\mu(E) r_k(x(E))^2}}{\sum_j e^{-\mu(E) r_j(x(E))^2}},
\]
the condition $w_k(E) r_k(x(E))^2\to0$ forces $\mu(E)\to+\infty$. This concludes the proof.
\end{proof}

The IVP \eqref{eq:IVP} is meaningful only where $J$ is invertible and
the right-hand side is sufficiently regular. We therefore study the domain
on which $F$ is smooth and identify the algebraic condition that ensures
invertibility of $J$.

Let $w = (w_1,\dots,w_m)^\top$ and denote by
\[
\Delta_m^\circ := \{\, w\in\RR^m : w_i>0\ \forall i,\ \sum_{i=1}^m w_i = 1 \,\}
\]
the interior of the probability simplex. 
\begin{proposition}
\label{prop:F-smooth}
Set 
\begin{equation}
\label{D-set}
\mathcal{D} := \RR\times\RR\times\Delta_m^\circ\times\RR^n.
\end{equation}
The mapping
\[
F: \; \mathcal{D} \;\to\; \RR^{2+m+n},
\qquad (\lambda,\mu,w^\top,x^\top)^\top\mapsto F(\lambda,\mu,w,x),
\]
is $C^\infty$. In particular all partial derivatives appearing in the
Jacobian $J$ are continuous on the open set $\mathcal{D}$.

\end{proposition}

\begin{proof}
The function $F$ (see \eqref{sys11}) is composed by polynomials, the map $w\mapsto \log w$, which is $C^\infty$ on $\Delta_m^\circ$), and linear algebra
operations involving $A$. 
\end{proof}

The following result isolates a convenient sufficient condition for
invertibility of $J$. The proof  relies only on the
Schur complement lemma.
\begin{proposition}[Sufficient conditions for invertibility of $J$]
\label{prop:J-invertible}
Let $y=(\lambda,\mu,w^\top,x^\top)^\top$ be such that $\mu>0$ and $w\in\Delta_m^\circ$. Denote by $V:=\mathrm{Range}(W^{1/2}A)\subseteq\RR^m$ the column space of $W^{1/2}A$,
and by $C:=\mathrm{span}\{u,r^2\}\subset\RR^m$ the space generated by the two constraint gradients $u$ and $r^2$. Assume the following:
\begin{itemize}
  \item[(a)] the two vectors $u$ and $r^2$ are linearly independent (equivalently, $r^2$ is not a constant vector);
  \item[(b)] the matrix 
\begin{equation}
\label{Bmatrix}
B(y):= I_m -2 \mu \diag(r^2)
\end{equation}
is positive definite on $V$, i.e.
\begin{equation}  
\label{B-pos}
v^\top B(y)\,v > 0, \quad\mbox{for all } v\in V\setminus\{0\}.
\end{equation}
\end{itemize}
Then  $J(y)$ is nonsingular and matrix $\widetilde {\cal H}$ defined at \eqref{Htmatrix} is positive definite.
\end{proposition}

\begin{proof} With reference to \eqref{J},  we first  show that the $2\times 2$ block matrix
$$
\widetilde J:=
\begin{pmatrix}
 W^{-1}          & 2\mu\,\diag(r)A \\
(\diag(r)A)^\top & A^\top W A
\end{pmatrix}
$$
is nonsingular, which  is tantamount to the invertibility of the Schur complement of matrix $W^{-1}$ in $\widetilde J$:
\begin{equation}
\label{Shmatrix}
\widehat S =  A^\top W A - 2\mu A^\top \diag(r) W \diag(r)A.
\end{equation}
We have
$$
\begin{array}{rcl}
\widehat S &=&  A^\top  W^{1/2} W^{1/2} A - 2\mu A^\top W^{1/2} \diag(r^2) W^{1/2} A \\[.15cm]
           &=&  (W^{1/2}A)^\top   (W^{1/2}A) - 2\mu (W^{1/2}A)^\top  \diag(r^2) (W^{1/2} A) \\[.15cm]
           &=&  (W^{1/2}A)^\top   \left( I_m - 2\mu  \diag(r^2)\right)  (W^{1/2} A) \\[.15cm]
           &=&  (W^{1/2}A)^\top   B(y)  (W^{1/2} A).       
\end{array}
$$
For any $z\in \RR^n\setminus\{0\}$, we have  $v:= (W^{1/2} A)z \not =0$, since $A$ has full column rank. Therefore, due to assumption (b), $\widehat S$ is positive definite and hence nonsingular. 

A comparison between matrices $\widehat S$ in \eqref{Shmatrix} and $S$ in \eqref{Smatrix} reveals that $S=2\mu \widehat S$. Since $\mu>0$, it follows that $S$ is positive definite and so is matrix $\widetilde {\cal H}$ defined in \eqref{Htmatrix}.

Now suppose $J(y)(\alpha,\beta,\gamma^\top,\eta^\top)^\top = 0$ for
scalars $\alpha,\beta$ and vector $\gamma\in\RR^m$ and $\eta \in \RR^n$. Writing the
four block-rows of this equality yields:
\begin{equation}
\label{four-equations}
\begin{array}{rl}
u^\top \gamma &= 0,\\[.15cm]
r^2{}^\top \gamma &= 0,\\[.15cm]
\alpha u + \beta r^2 + W^{-1}\gamma  + 2\mu \diag(r)A \eta &= 0, \\[.15cm]
(\diag(r)A)^\top \gamma  + A^\top W A \eta &= 0.
\end{array}
\end{equation}
Left-multiplying the third equation by $\gamma^\top$ and using the
first two gives 
\begin{equation}
\label{third-eq}
\gamma^\top W^{-1} \gamma +2\mu \gamma^\top (\diag(r) A) \eta =0.
\end{equation}
Analogously, left-multiplying the fourth equation by $2\mu \eta^\top$ gives 
\begin{equation}
\label{fourth-eq}
2\mu \eta^\top (\diag(r) A)^\top \gamma +2\mu \eta^\top A^\top W A \eta =0.
\end{equation}
Summing up \eqref{third-eq} and \eqref{fourth-eq} yields $(\gamma^\top,\eta^\top){\widetilde {\cal H}} (\gamma^\top,\eta^\top)^\top=0$ and, since  $\widetilde {\cal H}$ is positive definite, this implies $\gamma=0$ and $\eta=0$.
Consequently, the third equation in \eqref{four-equations} reduces to $\alpha  u+ \beta r^2=0$,
and by (a) this forces $\alpha=\beta=0$. Thus the kernel of $J(y)$
is trivial which means that $J(y)$ is invertible. 
\end{proof}

\subsection{Global continuation results}
\label{subsec_global_continuation}

The local analysis in Section~\ref{sec_local_results} guarantees the existence and uniqueness of a smooth solution branch $y(\MSE)$ of the stationarity system \eqref{Fsys1} for $\MSE$ sufficiently close to $\MSEuw$. In this section, we extend the discussion to the entire admissible range of the parameter $\MSE$, and characterize the possible behaviors of the branch as $\MSE$ decreases. The key observation is that, as long as the Jacobian $J(y)$ remains invertible and the vector field $f(y) = J(y)^{-1} e_2$ stays smooth within the open domain $\mathcal{D}$ defined in \eqref{D-set}, the solution can be prolonged by standard continuation arguments for autonomous differential equations. Conversely, the continuation may fail only if the solution leaves every compact subset of $\mathcal{D}$, or if $J(y)$ loses invertibility at a finite value of $\MSE$. The following theorem formalizes this dichotomy.

\begin{theorem}
\label{thm:max_interval}
Let $y(\MSE)$ be the solution of the IVP \eqref{eq:IVP} with initial condition $y(\MSEuw)=y^*$, defined on its maximal interval of existence $(\MSE_{\min},\MSEuw]$.
Then exactly one of the following alternatives holds:
\begin{enumerate}
  \item \emph{Finite breakdown:} $\MSE_{\min}>0$ (the solution cannot be continued below some positive value $\MSE_{\min}$);
  \item \emph{Extension to zero:} $\MSE_{\min}=0$ (the solution extends uniquely to $(0,\MSEuw]$).
\end{enumerate}
If finite breakdown occurs at $\MSE_{\min}>0$, then one of the following mutually non-exclusive behaviors must take place as $\MSE\downarrow \MSE_{\min}$:
\begin{enumerate}
  \item[(i)] $\|y(\MSE)\|\to+\infty$ (escape to infinity);
  \item[(ii)] $y(\MSE)$ remains bounded but approaches the boundary of $\mathcal D$, i.e.\ some weight $w_i(\MSE)\to 0$;
  \item[(iii)] $y(\MSE)$ remains bounded and inside $\mathcal D$ but $J(y(\MSE))$ becomes singular.
\end{enumerate}
\end{theorem}

\begin{proof}
The existence of a maximal interval $(\MSE_{\min},\MSEuw]$ follows from the Cauchy--Lipschitz theorem applied to the autonomous ODE $y' = f(y)$ with initial condition $y^*$. In our context, the general continuation theorem for ODEs (see e.g.\ \cite{Amann90}) states that if a maximal solution is defined on $(\MSE_{\min},\MSEuw)$, then either $\MSE_{\min}=0$ or the solution cannot be extended past $\MSE_{\min}>0$ because it either (a) escapes every compact subset of the domain or (b) approaches the boundary of the domain where $f$ is not defined. 
In more detail,  noting that the domain of $f$ is contained in $\mathcal D$  and that $f$ is defined only when $J(y)$ is invertible (see Proposition~\ref{prop:F-smooth}),  failure of continuation at $\MSE_{\min}>0$ can occur only if (i) $\|y(\MSE)\|\to\infty$, or (ii) $y(\MSE)$ approaches $\partial\mathcal D$ (which here is characterized by some weight reaching zero), or (iii) $J(y(\MSE))$ loses invertibility while $y(\MSE)$ remains in $\mathcal D$. 
\end{proof}

\section{Asymptotic analysis}
\label{sec_asymptotic}
Theorem~\ref{thm:max_interval} describes the three possible scenarios by which
the continuation of the solution branch may fail when approaching the lower
endpoint $\MSE_{\min}$. In order to refine this description, we now focus on a
stronger hypothesis concerning the weighted Gram matrix $A^\top W(\MSE)A$.
Recall that Proposition~\ref{prop:J-invertible} only requires positive
definiteness of $A^\top W(\MSE)A$ to ensure the invertibility of $J$.
Here, instead, we consider the slightly stronger \emph{uniform} definiteness condition stated in the following working assumption.
\begin{assumption}
\label{assumption1}
Let $y(\MSE)$ be the solution of the IVP \eqref{eq:IVP}. We assume that a constant $s_0>0$ exists such that
\begin{equation}
\label{eq:uniform-pd}
A^\top W(\MSE)A \succeq s_0^2 I_n
\qquad\text{for all }\MSE\in(\MSE_{\min},\MSEuw].
\end{equation} 
\end{assumption}
For a given matrix $Q$ of dimension $n$, the notation $Q \succeq s_0^2 I_n$ means that $x^\top Q x \ge s_0^2 ||x||^2$ for any vector $x \in \RR^n$. Note that \eqref{eq:uniform-pd} is equivalent to requiring that the smallest singular value
of the weighted design matrix remains uniformly bounded away from zero, namely
\begin{equation}
\label{eq:uniform-pd1}
\sigma_{\min}\!\bigl(W(\MSE)^{1/2}A\bigr) \ge s_0>0 \qquad\text{for all }\MSE\in(\MSE_{\min},\MSEuw].
\end{equation}
Under this hypothesis, we can exclude the first two alternatives of
Theorem~\ref{thm:max_interval}, so that loss of invertibility of $J$ becomes the
only mechanism by which a finite breakdown at $\MSE_{\min}>0$ may occur.

\begin{proposition}
\label{prop:exclude(i)-(ii)}
Assume $\MSE_{\min}>0$ in  Theorem~\ref{thm:max_interval}. Assume moreover the uniform positive definiteness
of the weighted normal matrix stated in Assumption \ref{assumption1}.
Then  the only possible finite breakdown case of Theorem~\ref{thm:max_interval} is (iii): loss of
invertibility of $J$.
\end{proposition}

\begin{proof}
We first prove uniform boundedness of $y(\MSE)=(\lambda(\MSE),\mu(\MSE),w(\MSE)^\top,x(\MSE)^\top)^\top$ which excludes alternative (i) in Theorem \ref{thm:max_interval}.

From Proposition~\ref{prop:mu_monotone} we know that $\mu(\MSE)$ is positive, increases as $\MSE$ decreases and diverges to infinity only if $\MSE\to 0^+$. Since $\MSE_{\min}>0$, we conclude that $\mu(\MSE)$ remains bounded  on $(\MSE_{\min},\MSEuw]$. For later use, we define $M_{\max}$ as an upper bound of  $\mu(\MSE)$ in the given existence interval: 
\[
\mu(\MSE)\le M_{\max}<+\infty\qquad\text{for all }\MSE\in(\MSE_{\min},\MSEuw].
\]
Consider now the normal system
\[
A^\top W(\MSE) A\, x(\MSE) = A^\top W(\MSE) b,
\]
that is, the stationarity condition arising from the third block-row of $F=0$.
By \eqref{eq:uniform-pd} the symmetric matrix $A^\top W(\MSE)A$
is uniformly positive definite with minimal eigenvalue at least $s_0^2$.
Therefore it is invertible for every $\MSE$ in the interval and
\[
x(\MSE) = \bigl(A^\top W(\MSE) A\bigr)^{-1} A^\top W(\MSE) b.
\]
Taking 2-norms and using the  relations $\|W(\MSE)\|\le 1$ and
$\|A^\top\| = \|A\|$ we obtain
\[
\|x(\MSE)\| \le \bigl\|\bigl(A^\top W(\MSE) A\bigr)^{-1}\bigr\|\,\|A^\top\|\,\|W(\MSE)b\|
\le \frac{\|A\|\,\|b\|}{s_0^2}.
\]
Hence, $x(\MSE)$ is uniformly bounded on $(\MSE_{\min},\MSEuw]$. 

To show boundedness of $\lambda(\MSE)$ we first analyze the residual vector $r(\MSE)$. Using $r(\MSE)=Ax(\MSE)-b$ and the bound of $\|x(\MSE)\|$ we get
\[
\|r(\MSE)\| \le \|A\|\,\|x(\MSE)\| + \|b\|
\le \|A\|\,\frac{\|A\|\,\|b\|}{s_0^2} + \|b\|
= \left(\frac{\|A\|^2}{s_0^2}+1\right)\|b\|.
\]
Thus the residual vector $r(\MSE)$ is uniformly bounded; in particular,
each scalar residual $r_i(\MSE)$ satisfies $0\le r_i(\MSE)^2\le R_{\max}$
for some finite $R_{\max}$ independent of $\MSE$.  From the normalization relation (see Lemma~\ref{lem:partition} or \eqref{w-expr}) we have 
$$
\exp(1+\lambda(\MSE))=Z(\MSE)=\sum_{j=1}^m e^{-\mu(\MSE) r_j(\MSE)^2}. 
$$
Combining the bounds
$\,e^{-M_{\max} R_{\max}}\le e^{-\mu r_j^2}\le 1$ for each term shows
\[
m e^{-M_{\max} R_{\max}} \le Z(\MSE)\le m.
\]
Taking logarithms gives
\[
\log m- M_{\max} R_{\max} \le \log Z(\MSE) \le \log m,
\]
hence
\[
\lambda(\MSE) = \log Z(\MSE) -1
\]
is uniformly bounded on $(\MSE_{\min},\MSEuw]$.

To exclude alternative (ii) of Theorem \ref{thm:max_interval}, we prove that no weight can vanish on $(\MSE_{\min},\MSEuw]$. Recall the explicit expression (Lemma~\ref{lem:partition})
\[
w_i(\MSE)=\frac{e^{-\mu(\MSE) r_i(\MSE)^2}}{Z(\MSE)}.
\]
From  the uniform bounds $r_i(\MSE)^2\le R_{\max}$ and
$\mu(\MSE)\le M_{\max}$ we obtain the uniform lower bound for each
numerator
\[
e^{-\mu(\MSE) r_i(\MSE)^2} \ge e^{-M_{\max} R_{\max}} > 0.
\]
As was observed above, since every term in the sum defining $Z(\MSE)$ is at most 1, we have
$Z(\MSE)\le m$. Therefore
\[
w_i(\MSE) \ge \frac{e^{-M_{\max} R_{\max}}}{m}=: \delta>0
\qquad\text{for all }i\ \text{and all }\MSE\in(\MSE_{\min},\MSEuw].
\]
All weights are uniformly bounded away from zero, thus no weight can vanish on the given interval.
In conclusion, if breakdown occurs at the finite level $\MSE_{\min}>0$, the only
remaining possibility of Theorem~\ref{thm:max_interval} is \emph{(iii)}, i.e.\
loss of invertibility of $J$.
\end{proof}

We now turn to the complementary scenario in which the branch satisfies $\MSE_{\min} = 0$, 
so that exact interpolation becomes asymptotically feasible as $\MSE\downarrow 0$.  
This corresponds to the situation where a nonempty subset of the data is perfectly compatible with the linear model and the MEWLS problem attempts to drive the corresponding residuals to zero.  
In this regime the behavior of the weights becomes markedly different: some of them converge to strictly positive limits, while the others are forced to vanish. These and further interesting asymptotic properties are analyzed below. 

\begin{lemma}\label{lem:core_index_set}
Assume that the uniform coercivity condition \eqref{eq:uniform-pd} 
holds for all sufficiently small $\MSE>0$. Define
\[
\varepsilon_0 := \frac{s_0^2}{\|A\|^2},
\qquad
\mathcal I(\MSE)
:= \bigl\{\, i \in \{1,\dots,m\} : w_i(\MSE) \ge \varepsilon_0 \,\bigr\}.
\]
Then, for all sufficiently small $\MSE>0$, the following properties hold true:
\begin{enumerate}
\item $\;|\mathcal I(\MSE)| \ge n$;

\item the submatrix of $A$ formed by the rows $\{a_i^\top : i \in \mathcal I(\MSE)\}$ has full column rank;

\item the weights in $\mathcal I(\MSE)$ are uniformly bounded away from zero:
\[
\inf_{\MSE\downarrow 0} \ \min_{i\in\mathcal I(\MSE)} w_i(\MSE) \ \ge\ \varepsilon_0 .
\]
\end{enumerate}
\end{lemma}

\begin{proof}
For every nonzero $v\in\mathbb{R}^n$, the uniform coercivity assumption 
\eqref{eq:uniform-pd} gives
\[
v^\top A^\top W(\MSE) A v 
= \sum_{i=1}^m w_i(\MSE)\,(a_i^\top v)^2 
\;\ge\; s_0^2 \|v\|^2 .
\]

Assume by contradiction that $|\mathcal I(\MSE)|<n$. Then the set of rows
$\{a_i^\top : i\in\mathcal I(\MSE)\}$ spans a strict subspace of
$\mathbb{R}^n$, so there exists $v\ne 0$ orthogonal to all of them. Thus
$a_i^\top v=0$ for $i\in\mathcal I(\MSE)$, therefore
\[
\sum_{i=1}^m w_i(\MSE)\,(a_i^\top v)^2
=\sum_{i\notin\mathcal I(\MSE)} w_i(\MSE)\,(a_i^\top v)^2
< \varepsilon_0 \sum_{i=1}^m (a_i^\top v)^2
= \varepsilon_0 \|Av\|^2
\le s_0^2 \|v\|^2 ,
\]
contradicting the coercivity bound. Hence $|\mathcal I(\MSE)|\ge n$, and the
corresponding rows of $A$ must have full column rank. The last statement
follows directly from the definition of $\mathcal I(\MSE)$.
\end{proof}

\begin{remark}\em
Lemma~\ref{lem:core_index_set} shows that, under the uniform coercivity
assumption \eqref{eq:uniform-pd}, for every sufficiently small
$\MSE>0$ there exists a ``core'' subset of indices
$\mathcal I(\MSE)$ of cardinality at least $n$ whose associated weights remain
uniformly positive and whose corresponding rows of $A$ retain full column rank.
In the data-fitting interpretation, these points continue to carry genuine
information along the entire continuation branch. As $\MSE\to 0^+$, they
necessarily determine the limiting model, which interpolates them while
effectively discarding all remaining observations. In particular, this
structural nondegeneracy prevents collapse of the weighted design matrix and
ensures identifiability throughout the continuation process. Example 1 in the 
next section illustrates this scenario.
\end{remark}

The next proposition exploits this structural property to characterize the
precise asymptotic distribution of the MEWLS weights as
$\MSE\downarrow 0$, showing in particular that the weights on the core
set become uniform in the limit.

\begin{proposition}
\label{prop:weights_on_S_uniform}
Assume that the branch of MEWLS stationary points
\[
y(\MSE)
=(\lambda(\MSE),\mu(\MSE),w(\MSE)^\top,x(\MSE)^\top)^\top
\]
extends down to $\MSE\downarrow 0$ (i.e., $\MSE_{\min}=0$), and that the
uniform coercivity condition \eqref{eq:uniform-pd} holds for all
sufficiently small $\MSE>0$. Define the limiting core index set
\[
S := \{\, i\in\{1,\dots,m\} : \liminf_{\MSE\to 0^+} w_i(\MSE) > 0 \,\},
\qquad s := |S| .
\]
The following statements hold true:
\begin{itemize}
\item[(a)] $\displaystyle S= {\cal I}(\MSE)$ for sufficiently small $\MSE$, hence $s\ge n$. 
\item[(b)] $\displaystyle i\in S ~ \Longleftrightarrow ~ \lim_{\MSE \to 0^+} r_i(\MSE)=0$.
\item[(c)] $\displaystyle  \lim_{\MSE\to 0^+} w_i(\MSE) =  0$,  for all $i\not \in S$.
\item[(d)] $\displaystyle  \lim_{\MSE\to 0^+} w_i(\MSE) = \frac{1}{s}$,  for all $i\in S$.
\item[(e)] The  solution component
$x(\MSE)$ converges as $\MSE\downarrow 0$ to the unique vector
$x^\ast\in\mathbb{R}^n$ satisfying the interpolation conditions on $S$,
i.e.
\[
A_S x^\ast = b_S,
\]
where $A_S$ denotes the submatrix of $A$ formed by the rows indexed by $S$ ($b_S$ is defined analogously),
and consequently for every $j=1,\dots,m$,
\[
\lim_{\MSE\downarrow 0} r_j(\MSE) = r_j^\ast := a_j^\top x^\ast - b_j.
\]
In particular $r_i^\ast=0$ for $i\in S$ and $r_j^\ast\neq 0$ for $j\notin S$.
\end{itemize}
\end{proposition}

\begin{proof}
(a). From the definition of $S$, 
there exist $c_0>0$ and $\MSE_1>0$ such that $w_i(\MSE)\ge c_0$ for all
$0<\MSE\le \MSE_1$. For such values of $\MSE$ it turns out that $S$ coincides with the set ${\cal I}(\MSE)$ introduced in Lemma \ref{lem:core_index_set}, therefore $s\ge n$.  

\smallskip

(b). As $\MSE \to 0^+$, the weighted mean squared error  satisfies
\[
0\le \sum_{i=1}^m w_j(\MSE)\, r_j(\MSE)^2
~ \longrightarrow ~ 0^+.
\]
Therefore
\begin{equation}
\label{wjrj2to0}
w_j(\MSE)\, r_j(\MSE)^2 ~ \longrightarrow ~ 0^+, \qquad \text{for all } j=1,\dots,m.
\end{equation}
Since $w_j(\MSE)\ge c_0$ on $S$, it follows that
$r_j(\MSE)^2\to 0$ for all $j\in S$. 

We now prove the converse: if $r_i(\MSE)^2\to 0$ then $i\in S$. Suppose,
by contradiction, that there exists $i$ with $r_i(\MSE)^2\to 0$ but
$\liminf_{\MSE\downarrow 0} w_i(\MSE)=0$ (so $i\notin S$). By the
definition of $\liminf$ there exists a subsequence $\MSE_k\downarrow 0$
such that $w_i(\MSE_k)\to 0$. 
Define the set
\[
V := \{\, \ell\notin S : r_\ell(\MSE_k)\to 0 \text{ as } k\to\infty \,\}.
\]
By construction, the set $V$ is nonempty since $i\in V$, and  may possibly contain other indices
not in $S$ whose residuals vanish along the same subsequence. By (a) 
we already know that $r_j(\MSE_k)\to 0$ for every $j\in S$, hence the
residuals tend to zero on the union $S\cup V$ along $\MSE_k$.

Consider the sequence of weight vectors $w(\MSE_k)\in\Delta_m$. By
compactness of the simplex there exists a subsequence (which we still denote by $\MSE_k$ for simplicity)
such that $w(\MSE_k)\to w^\ast\in\Delta_m$ as $k\to\infty$. By virtue of \eqref{wjrj2to0}, the non-vanishing
of the residuals outside $S\cup V$ forces $w^\ast_\ell = 0$ for every
$\ell\notin S\cup V$, hence $\operatorname{supp}(w^\ast)\subseteq S\cup V$
and $\sum_{j\in S\cup V} w^\ast_j =1$.

Next recall that each $w(\MSE_k)$ is an entropy maximizer among admissible
weights at level $\MSE_k$. Passing to the limit, $w^\ast$ must maximize
the entropy $H(w)=-\sum_i w_i\log w_i$ on the limiting feasible set,
which is the face of the simplex supported on $S\cup V$. Since $H$ is
strictly concave on that face, it has a unique maximizer, namely the
uniform distribution on $S\cup V$; in particular $w^\ast_\ell>0$ for
every $\ell\in V$. But this contradicts the assumption that
$\ell\notin S$, because along the subsequence we found a positive limit for $w_\ell$.
Therefore, no such index $i\notin S$ can exist, and consequently for every
$\ell\notin S$ we must have $\liminf_{\MSE\downarrow 0} r_\ell(\MSE)>0$.

\smallskip

(c). We now show that for every $j\notin S$ we actually have 
$w_j(\MSE)\to 0$ as $\MSE \to 0^+$. This is a direct consequence of \eqref{wjrj2to0} after noting that the condition $\liminf_{\MSE\downarrow 0} r^{2}_\ell(\MSE)>0$ implies that $r_\ell(\MSE)$ is bounded away from zero for sufficiently small $\MSE$.

\smallskip

(d). It remains to prove that the weights on $S$ converge to the
uniform value $1/s$. From (c) we already know that
$\sum_{j\notin S} w_j(\MSE)\to 0$, hence
\[
\sum_{i\in S} w_i(\MSE) \longrightarrow 1 \qquad(\MSE\downarrow 0).
\]
By compactness of the simplex, for any sequence $\MSE_k\downarrow 0$ the
vectors $w(\MSE_k)$ admit a convergent subsequence; by (c) any limit
point $w^\ast$ is supported on $S$ and satisfies $\sum_{i\in S} w^\ast_i = 1$.
Thus every cluster point of $w(\MSE)$ belongs to the face
$\mathcal F_0$ of the simplex determined by $S$, namely
$$
\mathcal F_0 := \{ w\in\Delta_m : \operatorname{supp}(w)\subseteq S \}.
$$
We now use the variational characterization: for each $\MSE>0$ the
vector $w(\MSE)$ maximizes the entropy $H(w)=-\sum_i w_i\log w_i$ on
the admissible set. Passing to the limit, any cluster point $w^\ast$ must
maximize $H$ on the limiting feasible set $\mathcal F_0$. But $H$ is
strictly concave on $\mathcal F_0$ and hence has a unique maximizer,
namely $w^\ast_i = 1/s$, $i\in S$. Since every cluster point coincides with the same vector, the entire
family $w(\MSE)$ converges to the uniform distribution on $S$, i.e.
$$
\lim_{\MSE \to 0^+} w_i(\MSE) = \frac{1}{s},\qquad i\in S.
$$

\smallskip

(e). By (c) and (d) the weight vector converges componentwise:
\[
w_i(\MSE) \longrightarrow  w^*_i=
\begin{cases}
1/s, & i\in S,\\[4pt]
0, & i\notin S,
\end{cases}
\qquad \text{as } \MSE\to 0^+.
\]
As usual, let $W(\MSE)=\diag(w_1(\MSE),\dots,w_m(\MSE))$ and define the limiting diagonal
matrix $W^\ast := \diag(w^*_1,\dots, w^*_m)$.
Consider the normal matrix $M(\MSE):=A^\top W(\MSE) A$. From the componentwise
convergence $W(\MSE)\to W^\ast$ we have $M(\MSE)\to M^\ast:=A^\top W^\ast A$.
But
\[
M^\ast = \frac{1}{s}\, A_S^\top A_S,
\]
and since the rows of $A_S$ have full column rank (Lemma~\ref{lem:core_index_set}),
the matrix $A_S^\top A_S$ is positive definite; hence $M^\ast$ is positive
definite. Consequently there exists $\MSE_0>0$ such that $M(\MSE)$ is
invertible for all $0<\MSE\le \MSE_0$, and $M(\MSE)^{-1}\to (M^\ast)^{-1}$
as $\MSE\downarrow 0$.

For such $\MSE$ the vector $x(\MSE)$ is given by the normal equations
solution
\[
x(\MSE) = M(\MSE)^{-1} A^\top W(\MSE) b.
\]
Passing to the limit $\MSE\downarrow 0$ and using the convergence
$W(\MSE)\to W^\ast$ and $M(\MSE)^{-1}\to (M^\ast)^{-1}$ yields
\[
x(\MSE) \longrightarrow x^\ast := (M^\ast)^{-1} A^\top W^\ast b.
\]
Observing that $A^\top W^\ast b = (1/s) A_S^\top b_S$ and that
$(M^\ast)^{-1} = s (A_S^\top A_S)^{-1}$, one readily checks that
$x^\ast$ is exactly the unique solution of $A_S x = b_S$.

Finally, continuity of the linear maps $x\mapsto a_j^\top x$ implies
\[
r_j(\MSE)=a_j^\top x(\MSE)-b_j \longrightarrow a_j^\top x^\ast - b_j =: r_j^\ast,
\]
for every $j=1,\dots,m$. Since $A_S x^\ast=b_S$, we have $r_i^\ast=0$
for $i\in S$; for $j\notin S$ the limit $r_j^\ast$ cannot vanish
(otherwise $j$ would belong to $S$ by definition), hence
$r_j^\ast\neq 0$. This proves (e) and completes the proof of the lemma.
\end{proof}

By Proposition~\ref{prop:weights_on_S_uniform} the branch, if it extends to
$\MSE\downarrow 0$, has a well defined limiting structure:
there exists a core index set $S$ (with $s:=|S|\ge n$) such that the
limit weight vector $w^*$ satisfies $w_i^*=1/s$ for $i\in S$ and
$w_j^*=0$ for $j\notin S$, and the corresponding limiting parameter
vector $x^*$ interpolates the observations on $S$. 

In the next two lemmas, under Assumption \ref{assumption1}, we show the following asymptotic behaviors:  
$$
r_i(\MSE)=O(\MSE), ~ \text{for } i\in S; \qquad w_j(\MSE)=O(\MSE), ~ \text{for } j\notin S; \qquad  \mu(\MSE)=O(|\log\MSE|).
$$

\begin{lemma}\label{lem:weights_residuals_mu}
Assume the hypotheses and notations of Proposition \ref{prop:weights_on_S_uniform}. 
Then there exist positive constants $C_1,C_2$ and $\MSE_0>0$ such that, for all $0<\MSE<\MSE_0$, the following estimates hold:

\begin{enumerate}
\item[(a)] For every index $j\in S^c$ one has
\begin{equation}
\label{eq:wj-bound}
w_j(\MSE) \le C_1\,\MSE \qquad\text{(hence } w_j(\MSE)=O(\MSE)\text{ as }\MSE\to0^+).
\end{equation}
\item[(b)] The primal variable satisfies the linear estimate
\begin{equation}
\label{eq:Deltax-bound}
\|x(\MSE)-x^\ast\| \le C_2\,\MSE,
\end{equation}
where $x^\ast =\lim_{\MSE\to 0^+} x(\MSE)$. Consequently, for every $i\in S$,
\begin{equation}
\label{eq:ri-bound}
r_i(\MSE) = O(\MSE).
\end{equation}
\end{enumerate}
\end{lemma}

\begin{proof}
(a). Fix $j\in S^c$. By Proposition \ref{prop:weights_on_S_uniform}(b),  there exists $\delta>0$ and $\MSE_1>0$ such that for all $0<\MSE<\MSE_1$ we have $r_j(\MSE)^2\ge \delta$. Hence, from
\[
w_j(\MSE)\, r_j(\MSE)^2 \le \sum_{i=1}^m w_i(\MSE) r_i(\MSE)^2 = \MSE,
\]
we obtain
\[
w_j(\MSE) \le \frac{\MSE}{r_j(\MSE)^2} \le \frac{\MSE}{\delta}.
\]
Thus \eqref{eq:wj-bound} holds with $C_1:=1/\delta$ (and $\MSE_0\le \MSE_1$).

\smallskip

(b). By reordering rows of $A$ if necessary, write
\[
A = \begin{pmatrix} A_S \\[2pt] A_C \end{pmatrix},\qquad
r(\MSE)=\begin{pmatrix} r_S(\MSE) \\[2pt] r_C(\MSE) \end{pmatrix},\qquad
W(\MSE)=\begin{pmatrix} W_S(\MSE) & 0 \\ 0 & W_C(\MSE) \end{pmatrix},
\]
where $A_S$ has the rows indexed by $S$ and $A_C$ the rows indexed by $S^c$.
The stationarity condition (normal equations) reads
\[
A^\top W r = 0,
\]
which, in block form, becomes
\begin{equation}\label{eq:block_stationarity}
A_S^\top W_S r_S = -A_C^\top W_C r_C.
\end{equation}

We will show that the right-hand side  is $O(\MSE)$, and then
invertibility and coercivity on the $S$-block yields a linear bound on $x-x^\ast$.

First note that for $j\in S^c$ we have $r_j(\MSE)$ bounded and, by part (a), $w_j(\MSE)=O(\MSE)$.
Hence componentwise $(W_C r_C)_j = w_j(\MSE) r_j(\MSE) = O(\MSE)$, therefore
\[
\|W_C r_C\| \le C_3\,\MSE
\]
for some constant $C_3>0$ and all sufficiently small $\MSE$. Consequently,
\begin{equation}
\label{eq:est1}
\|A_C^\top W_C r_C\| \le \|A_C^\top\|\,\|W_C r_C\| \le \|A_C^\top\| C_3\,\MSE =: C_4\,\MSE.
\end{equation}

Next we produce a lower bound for the left-hand side $A_S^\top W_S r_S$ in terms of
$\|x(\MSE)-x^\ast\|$. Since $r_S(\MSE)=A_S x(\MSE)-b_S$ and $r_S(x^\ast)=0$ by definition of $S$,
we have $r_S(\MSE)=A_S \Delta x$ where $\Delta x := x(\MSE)-x^\ast$. Therefore
\[
A_S^\top W_S r_S = A_S^\top W_S A_S \Delta x.
\]
By the global coercivity \eqref{eq:uniform-pd} we have $A^\top W A \succeq s_0^2 I$, and in particular
for any vector $z\in\mathbb{R}^n$,
$z^\top A^\top W A z \ge s_0^2 \|z\|^2$. Choose $z=\Delta x$. Then
\[
\Delta x^\top A^\top W A \Delta x \ge s_0^2 \|\Delta x\|^2.
\]
Expanding the left-hand side in blocks yields
\[
\Delta x^\top A^\top W A \Delta x
= \Delta x^\top \big( A_S^\top W_S A_S + A_C^\top W_C A_C \big)\Delta x.
\]
In particular,
\[
\Delta x^\top A_S^\top W_S A_S \Delta x
\ge s_0^2 \|\Delta x\|^2 - \Delta x^\top A_C^\top W_C A_C \Delta x.
\]
We bound the right-most perturbation term using operator norms:
\[
\big|\Delta x^\top A_C^\top W_C A_C \Delta x\big|
\le \|A_C^\top W_C A_C\|\,\|\Delta x\|^2
\le \|A_C\|^2 \,\|W_C\| \,\|\Delta x\|^2,
\]
where $\|W_C\|=\max_{j\in S^c} w_j(\MSE)=O(\MSE)$ by part (a). Hence, for
$\MSE$ sufficiently small the perturbation term is bounded by $(s_0^2/2)\|\Delta x\|^2$,
so that
\[
\Delta x^\top A_S^\top W_S A_S \Delta x \ge \frac{s_0^2}{2}\,\|\Delta x\|^2.
\]
Therefore, due to the arbitrary nature of the perturbation $\Delta x$, the minimal eigenvalue $\lambda_{\min}(A_S^\top W_S A_S)$ is bounded
below by $c_0:=s_0^2/2>0$ for all sufficiently small $\MSE$. Consequently,
\begin{equation}
\label{eq:est2}
\|A_S^\top W_S A_S \Delta x\| \ge c_0 \|\Delta x\|.
\end{equation}
Combining \eqref{eq:block_stationarity} with the estimates \eqref{eq:est1} and \eqref{eq:est2} we obtain
\[
c_0 \|\Delta x\| \le \|A_S^\top W_S A_S \Delta x\| = \|A_C^\top W_C r_C\| \le C_4\,\MSE.
\]
Hence
\[
\|\Delta x\| = \|x(\MSE)-x^\ast\| \le \frac{C_4}{c_0}\,\MSE.
\]
This proves \eqref{eq:Deltax-bound} with $C_2:=C_4/c_0$. Finally, since $r_S(\MSE)=A_S\Delta x$, we deduce
\[
\|r_S(\MSE)\| \le \|A_S\|\,\|\Delta x\| \le \|A_S\| C_2\,\MSE,
\]
and componentwise $r_i(\MSE)=O(\MSE)$ for each $i\in S$, proving \eqref{eq:ri-bound}. 
\end{proof}
%%%%%%%%%%%

\begin{lemma}\label{lemma:mu-log}
Under the hypotheses of Lemma~\ref{lem:weights_residuals_mu} 
the following properties hold true as $\MSE\to0^+$:
\begin{enumerate}
\item $\displaystyle \mu(\MSE)=O\!\big(|\log(\MSE)|\big)$.
\item For every $i\in S$, $\displaystyle \mu(\MSE)\,r_i(\MSE)^2 \longrightarrow 0$.
\end{enumerate}
\end{lemma}

\begin{proof}
We keep the notation $\mu=\mu(\MSE)$ and $r=r(\MSE)$ throughout this proof. By
Lemma~\ref{lem:weights_residuals_mu} we have the uniform bounds
(for $\MSE$ small)
\[
r_i(\MSE)=O(\MSE)\quad(i\in S),\qquad w_j(\MSE)=O(\MSE)\quad(j\in S^c).
\]
In particular, there exists $C_r>0$ and $\MSE_0>0$ such that for $0<\MSE<\MSE_0$ and
every $i\in S$,
\begin{equation}
\label{eq:ri-bound1}
r_i(\MSE)^2 \le C_r\,\MSE^2.
\end{equation}

Let $\delta>0$ be the uniform lower bound on squared residuals on the inactive set, i.e.\ for $j\in S^c$ and $\MSE$ sufficiently small
\begin{equation}
\label{eq:rj-bound2}
r_j(\MSE)^2 \ge \delta.
\end{equation}
Denote $s:=|S|\ge n$. Recall the normalization 
\[
Z(\MSE)=\sum_{k=1}^m e^{-\mu(\MSE) r_k(\MSE)^2},
\qquad w_k(\MSE)=\frac{e^{-\mu(\MSE) r_k^2(\MSE)}}{Z(\MSE)}.
\]

First of all, we determine a lower bound for $Z(\MSE)$. By \eqref{eq:ri-bound1} we have for all $i\in S$ and $\MSE<\MSE_0$,
\[
e^{-\mu r_i^2} \ge e^{-\mu C_r \MSE^2}.
\]
Summing over $i\in S$ yields
\[
\sum_{i\in S} e^{-\mu r_i^2} \ge s\, e^{-\mu C_r \MSE^2}.
\]
Therefore the partition function satisfies
\begin{equation}
\label{eq:Zlb3}
Z(\MSE) \;=\; \sum_{i\in S} e^{-\mu r_i^2} + \sum_{j\notin S} e^{-\mu r_j^2}
\;\ge\; s\, e^{-\mu C_r \MSE^2}.
\end{equation}

Now we determine an upper bound of the mean squared error, after splitting it in the two contributions coming from the core set $S$ and its complimentary $S^C$:
\begin{equation}
\label{eq:MSE-complete}
\MSE \;=\;  \sum_{i\in S} w_i r_i^2 + \sum_{j\notin S} w_j r_j^2. 
\end{equation}
For the active set, using \eqref{eq:ri-bound1} and the boundedness of the $w_i(\MSE)$'s
on $S$ (they converge to positive limits), there exists $C_S>0$ with
\begin{equation}
\label{eq:activeset7}
\sum_{i\in S} w_i r_i^2 \le C_S\, \MSE^2.
\end{equation}
For $j\notin S$, exploiting \eqref{eq:rj-bound2}, we get
\[
\sum_{j\notin S} w_j r_j^2
= \sum_{j\notin S} \frac{e^{-\mu r_j^2}}{Z(\MSE)}\, r_j^2
\le \frac{1}{Z(\MSE)} \sum_{j\notin S} r_j^2\, e^{-\mu \delta}.
\]
Let $R^2(\MSE):=\max_{j\notin S} r_j^2(\MSE)$, which is a finite quantity bounded away from zero, independently of the choice of $\MSE\in(0,\MSE_0)$. Then
\begin{equation}
\label{eq:inactiveset4}
\sum_{j\notin S} w_j r_j^2 \le \frac{(m-s) R^2}{Z(\MSE)}\, e^{-\mu\delta}.
\end{equation}
Combine \eqref{eq:Zlb3} and \eqref{eq:inactiveset4} to obtain
\begin{equation}
\label{eq:activeset5}
\sum_{j\notin S} w_j r_j^2 \le \frac{(m-s) R^2}{s}\, e^{-\mu(\delta - C_r \MSE^2)}, \qquad 0<\MSE<\MSE_0.
\end{equation}
Choose $\MSE_1\in(0,\MSE_0)$ so small that for $0<\MSE<\MSE_1$ we have
$C_r \MSE^2 \le \delta/2$. Then for such $\MSE$,
\[
\delta - C_r \MSE^2 \ge \delta/2 >0,
\]
and \eqref{eq:activeset5} gives the exponential bound
\begin{equation}
\label{eq:activeset6}
\sum_{j\notin S} w_j r_j^2 \le C'\, e^{-\mu(\delta/2)}, \qquad 0<\MSE<\MSE_0,
\end{equation}
for some constant $C'>0$ independent of $\MSE$.

Using \eqref{eq:activeset6} and \eqref{eq:activeset7}, from \eqref{eq:MSE-complete} we deduce that for $0<\MSE<\MSE_1$,
\[
\MSE \le C_S \MSE^2 + C' e^{-\mu(\delta/2)}.
\]
For $\MSE$ small, the factor $1- C_S \MSE$ is bounded below by $1/2$; hence, for such values of $\MSE$, 
\[
\frac{\MSE}{2} \le C' e^{-\mu(\delta/2)}.
\]
This implies
\[
e^{\mu(\delta/2)} \le \frac{2C'}{\MSE},
\]
and taking logarithms yields the desired upper bound
\[
\mu \le \frac{2}{\delta}\log\!\bigg(\frac{2C'}{\MSE}\bigg) = O(\log(1/\MSE)).
\]

Finally, combining this result with \eqref{eq:ri-bound1}, for any fixed $i\in S$,  we have 
\[
\mu(\MSE)\, r_i(\MSE)^2 \le C_r\, \mu(\MSE)\, \MSE^2 = O\!\big(\log(1/\MSE)\cdot \MSE^2\big). \qedhere
\]
\end{proof}

%%%%%%%%%%%%%%

We now investigate the behavior of the matrix
\[
B(y) = I_m - 2\mu(\MSE)\,\diag\!\bigl(r_1(\MSE)^2,\dots,r_m(\MSE)^2\bigr)
\]
 in the asymptotic regime $\MSE\downarrow 0$. In Proposition \ref{prop:J-invertible}, this matrix played a central role in establishing sufficient conditions for the invertibility of the Jacobian \eqref{J} and the positive definiteness of the matrix $\widetilde H$ defined in \eqref{Htmatrix}. We recall that these properties are essential for the well-posedness of the differential problem and for concluding that the points along the resulting solution branch are indeed entropy maximizers. Interestingly, such conditions also turn out to be necessary: this is the content of the following proposition.

\begin{proposition}\label{prop:B_positive_on_V}
Under the hypotheses of Proposition~\ref{prop:weights_on_S_uniform}, let
\[
B(\MSE) := I - 2\,\mu(\MSE)\,\diag\bigl(r(\MSE)^2\bigr) \quad \text{and} \quad V:=\operatorname{range}(W^{1/2}A).
\]
There exists $\MSE_0>0$ such that for all $0<\MSE<\MSE_0$ the restriction of $B(\MSE)$
to $V$ is positive definite, namely
\[
\exists\,\MSE_0>0:\qquad
\forall\,0<\MSE<\MSE_0,\quad v^\top B(\MSE) v >0\quad\text{for all }v\in V\setminus\{0\}.
\]
\end{proposition}

\begin{proof}
We use the Rayleigh--quotient characterization on $V$. Let $P_V$ denote the orthogonal
projector onto $V$. The restriction of $B$ to $V$ is positive definite iff
\[
2\mu(\MSE)\,\lambda_{\max}\!\big(P_V\,\diag(r^2)\,P_V\big) < 1,
\]
where $\lambda_{\max}(\cdot)$ denotes the largest eigenvalue. Since any $v\in V$ may be cast as $v=W^{1/2}A z$ for some $z$, we conclude that
\[
\lambda_{\max} = \max_{z\neq0}
\frac{z^\top A^\top \big(\diag(w\,r^2)\big) A z}{z^\top A^\top W A z}.
\]
Hence,
\begin{equation}
\label{eq:lambda_bound}
\lambda_{\max} \le \frac{\|A\|^2\,\|\diag(w r^2)\|}{s_0^2}
= \frac{\|A\|^2}{s_0^2}\,\max_{j} \bigl(w_j r_j^2\bigr).
\end{equation}
We split the maximum over indices in $S$ and $S^c$.

For $i\in S$, we have $w_i(\MSE)\to 1/s$ from Proposition \ref{prop:weights_on_S_uniform}, and 
$\mu r_i^2\to0$ from Lemma \ref{lemma:mu-log}. Therefore
\[
2\mu(\MSE)\,w_i(\MSE)\,r_i(\MSE)^2 = 2 w_i(\MSE)\,(\mu(\MSE) r_i^2(\MSE))
\longrightarrow 0\quad (i\in S).
\]

For $j\notin S$ we consider the bound \eqref{eq:activeset6} derived in Lemma \ref{lemma:mu-log} and, using the same notations, we deduce that
$$
 w_j(\MSE) r_j^2(\MSE) \le C'\, e^{-\mu(\MSE)(\delta/2)} \qquad \text{($\MSE$ small)}.
$$
Consequently,
\[
2\mu(\MSE)\, w_j(\MSE)\, r_j(\MSE)^2
\le 2\mu(\MSE)\, C'\, e^{-\mu(\MSE)(\delta/2)}.
\]
Using the elementary fact that for any $a>0$ the function $t\mapsto t e^{-a t}$ tends to $0$ as $t\to+\infty$,
and since $\mu(\MSE)\to+\infty$ as $\MSE\to0$, the right-hand side vanishes as $\MSE\to0^+$.
Hence the contribution from every $j\notin S$ also satisfies
\[
2\mu(\MSE)\,w_j(\MSE)\,r_j(\MSE)^2 \longrightarrow 0.
\]

Combining the two cases and using the bound \eqref{eq:lambda_bound}, we obtain
\[
2\mu(\MSE)\,\lambda_{\max}\!\big(P_V\diag(r^2)P_V\big)
\le \frac{2\mu(\MSE)\|A\|^2}{s_0^2}\max_j (w_j r_j^2)
\longrightarrow 0 \qquad \text{as } \MSE\to0^+.
\]
Therefore, there exists $\MSE_0>0$ such that for $0<\MSE<\MSE_0$ the left-hand side is strictly less than $1$,
hence $B(\MSE)$ is positive definite on $V$ for all such $\MSE$.
This proves the proposition.
\end{proof}
%%%%%%%%%%%%%%%%%%%%%%%%%%%%%
\section{Numerical Illustrations}
\label{sec_numerics}
The aim of this section is primarily methodological. 
The numerical experiments are designed to illustrate and validate the theoretical 
properties established in the previous sections. 
Applications of the maximum-entropy weighted least squares  framework 
to more challenging and real-world datasets, as well as comparisons with state-of-art robust techniques, have already been discussed in previous works \cite{AmBrIa26,AmFaGrIaLaMaNoRu,BrGiIaRu24,DeFaIaLoMaRu25,GiIa21}. 
Here we deliberately consider controlled synthetic examples and a simple model  that allow  a transparent interpretation of the results,  highlighting  the structural behavior of the solution path $E \mapsto (w(E),x(E))$ 
and visualizing the behavior predicted by the theory. All numerical experiments were performed in Matlab\textsuperscript{\textregistered}  (version R2024b) on a computer equipped with a 3.6 GHz Intel i9 processor and 32 GB of RAM. 

From the computational point of view, the differential system obtained in 
Section \ref{subsec_IVP} is treated as a mass-matrix ODE. 
Rather than explicitly inverting the Jacobian matrix $J$, we retain it as 
a mass matrix and solve the resulting system using the Matlab solver 
\texttt{ode23t}. Since this solver handles initial value problems but not 
final value problems, we introduce the change of variable $t = -E$, 
so that the evolution with decreasing MSE is reformulated as a standard IVP. 
The integration is therefore performed forward in $t$, which corresponds to 
moving from larger to smaller values of $E$ along the global solution branch. 
The resulting differential problem takes the form
\begin{equation}
\label{FDE_final} \left\{
\begin{array}{ll}
J(y(t))\, y'(t) = -e_2, & t\in [-\MSEuw,-\MSE_{\mathrm f}], \\[2mm]
y(-\MSEuw)=y^\ast,
\end{array} \right.
\end{equation} 
where $y(t)=(\lambda(t),\mu(t),w(t),x(t))^\top$, $\MSE_{\mathrm f}$ is the user desired MSE value, and $y^\ast$ denotes the 
uniform-weight solution. 

However, in the figures accompanying each example below, we display the 
solution dynamics with respect to $\MSE$ rather than $t$, in order to 
emphasize the natural parametrization of the problem.

We emphasize that the solution strategy adopted here, based on the 
integration of the associated initial value problem, represents a novel 
computational approach within the MEWLS framework. In previous works, 
the nonlinear system \eqref{sys1} was solved by progressively reducing the 
MSE through empirical continuation strategies.  In contrast, the present formulation embeds the variation of the MSE 
directly into the differential structure of the problem, so that the 
step-size adaptation is intrinsically handled by the ODE solver. 
This leads to a more systematic and dynamically consistent exploration 
of the global solution branch.

We finally remark that the Jacobian matrix $J(y)$ in \eqref{FDE_final}  exhibits a highly sparse 
structure, reflecting the intrinsic coupling pattern of the underlying 
optimality system. Exploiting this sparsity is crucial for large-scale 
implementations. However, the design of an efficient and fully structured 
numerical treatment requires a dedicated and nontrivial analysis, 
which falls outside the scope of the present paper and will be 
investigated in future work.

\subsection{Example 1}

We consider a simple overdetermined linear system arising from 
a regression problem in $\mathbb{R}^2$. Ten points (blue circles in Figure~\ref{fig1}) are uniformly distributed along the line
\begin{equation}
\label{model1}
y=\tfrac12 x, \qquad x\in[0,1].
\end{equation}
These points represent the \emph{inliers}.  We then add ten further points having the same abscissae but with ordinates 
randomly distributed in the upper part of the square $[0,1]\times[0,1]$. 
These additional data (red asterisks in Figure~\ref{fig1}) are clear \emph{outliers}, 
as they significantly deviate from the original linear trend.

\begin{figure}
\begin{center}
\includegraphics[width=.85\textwidth]{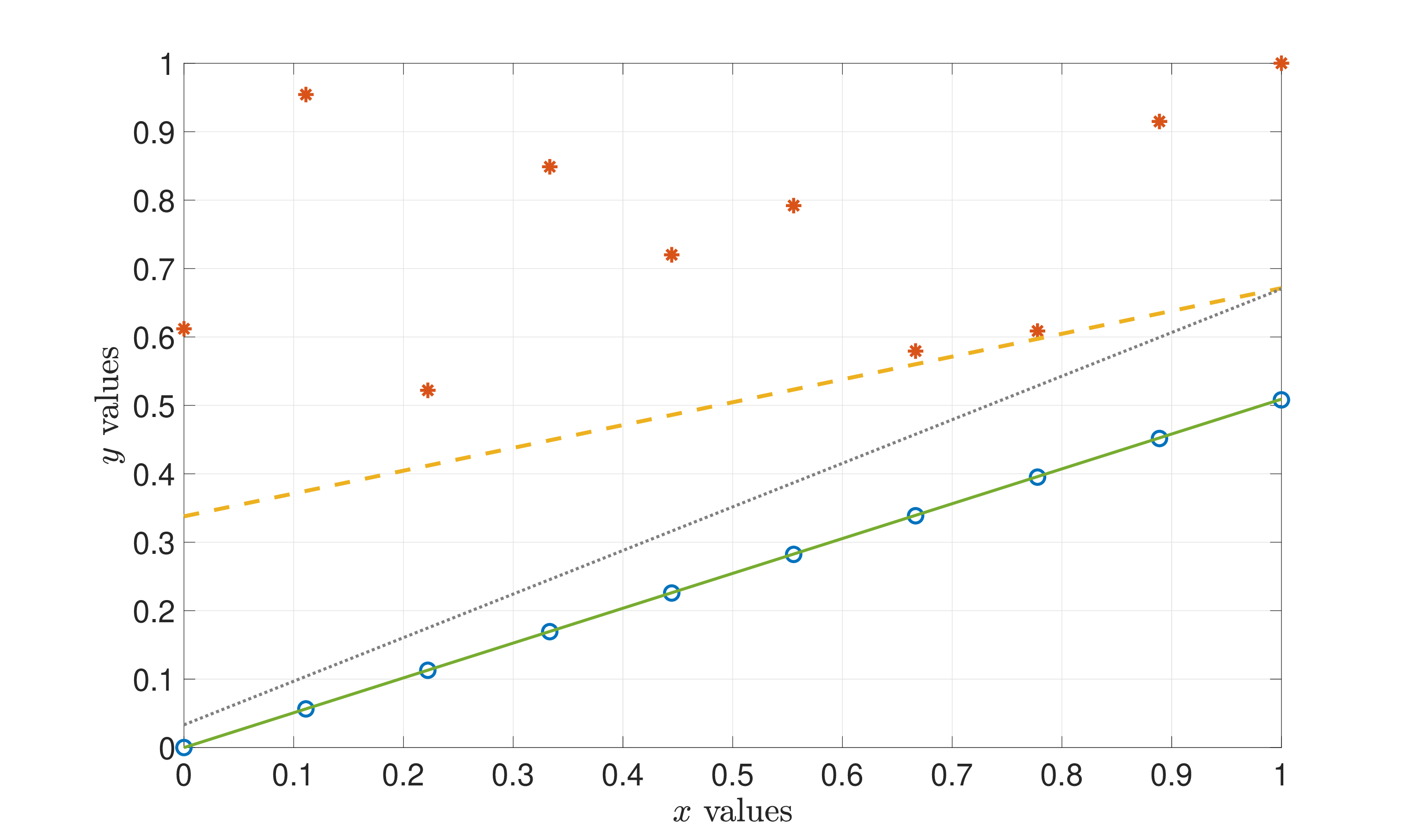} 
\caption{ Blue circles: inliers lying exactly on the line $y=\tfrac12 x$. 
Red asterisks: outliers.  Yellow dashed line: OLS regression with uniform weights on the full dataset. 
Green solid line: MEWLS regression at the final MSE level, coinciding with the outlier-free solution. 
Dotted gray line: intermediate MEWLS configuration corresponding to $E=3\cdot 10^{-2}$.  
\label{fig1}}
\end{center}
\end{figure}

The OLS regression line computed with uniform weights on the full 
20-point dataset (yellow dashed line in Figure~\ref{fig1}) is visibly attracted 
by the outliers and deviates substantially from the line determined by 
the inliers alone. We initialize the IVP at the final value
$ E = \mathrm{MSE}_{\mathrm{uw}},$
corresponding to the uniform-weight solution, and integrate the system 
down to $E = 1\times 10^{-4}$, a value chosen for graphical clarity.

Figure~\ref{fig2} shows the evolution of the twenty weights as $E$ decreases. 
The plot must be read from right to left, since decreasing $E$ corresponds 
to forward integration in the transformed time variable $t=-E$.

\begin{figure}
\begin{center}
\includegraphics[width=.85\textwidth]{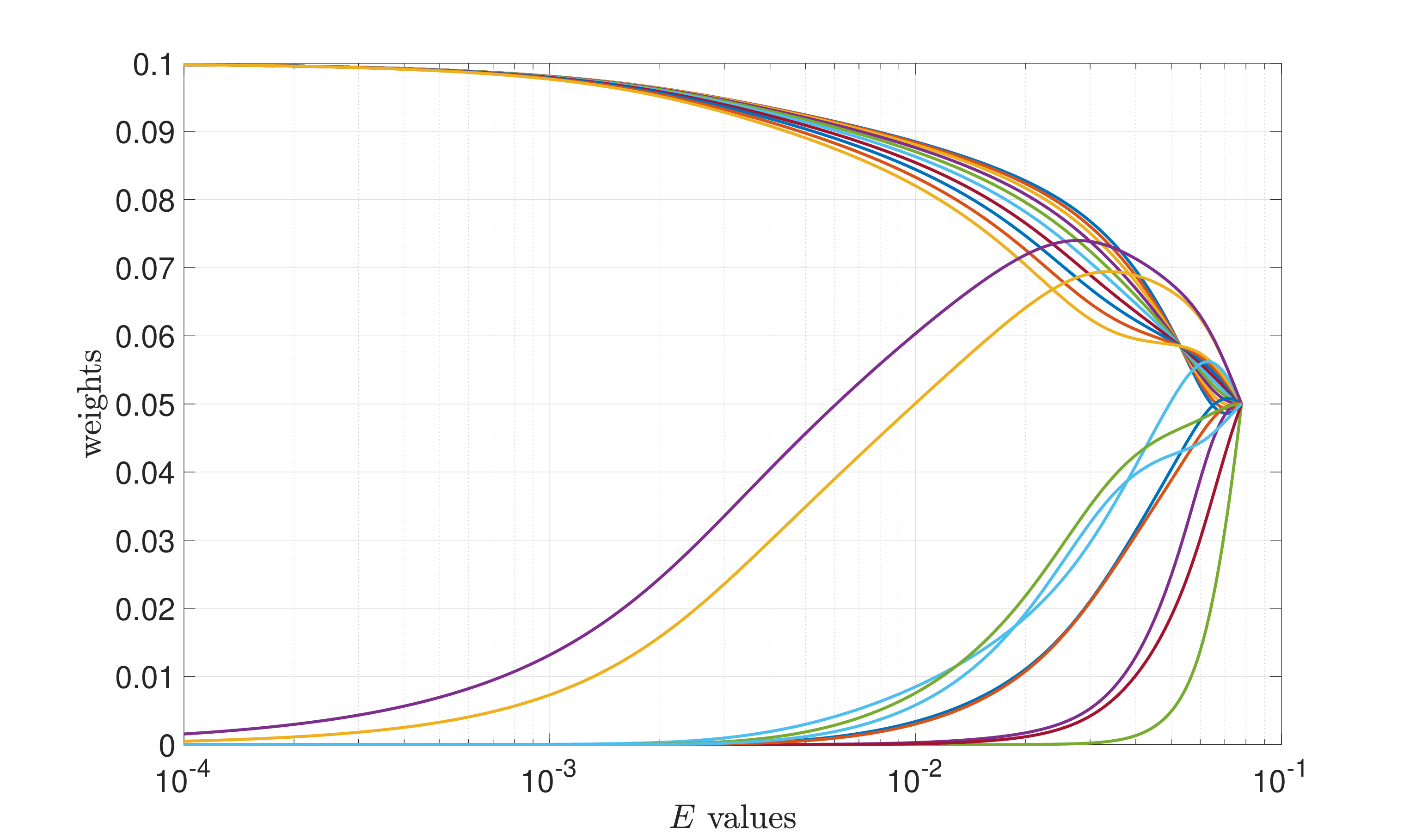} 
\caption{Evolution of the twenty weights $w_i(E)$ as the MSE decreases 
(from right to left).  The weights associated with the outliers decay to zero, while the 
remaining weights converge to the uniform distribution $1/10$, 
in agreement with Proposition~\ref{prop:weights_on_S_uniform}.  
\label{fig2}}
\end{center}
\end{figure}

In this first experiment the inliers lie \emph{exactly} on the model line 
\eqref{model1}. Hence, the assumptions of Proposition~\ref{prop:weights_on_S_uniform} 
are fully satisfied: there exists a subset $S$ of indices corresponding to 
vanishing residuals, and the theory predicts a uniform distribution of the 
weights over $S$ in the limit configuration.  This is precisely what is observed numerically:
\begin{itemize}
\item the ten weights associated with the outliers decay towards zero;
\item the remaining ten weights converge to the constant value $1/10$.
\end{itemize}

Thus, in the limit configuration, the method automatically restores 
a uniform weighting on the inlier subset while suppressing the influence 
of the outliers. As a consequence, the final regression line coincides, 
up to numerical accuracy, with the one obtained from the original 
10-point dataset without outliers, and is displayed as the green solid line 
passing through the blue circles in Figure~\ref{fig1}.

It is worth noting that the decay of the outlier weights is not uniform. 
Two weights, corresponding to points relatively close to the inlier line, 
initially increase together with the inlier weights, reaching a maximum at approximately
$E = 3 \cdot 10^{-2}$, represented by the gray dotted regression line in Figure~\ref{fig1}. Only when $E$ is reduced below a critical threshold does the entropy–MSE 
balance become incompatible with their inclusion in the effective inlier set. 
At that stage, the constraint on the MSE enforces a redistribution of mass, 
and these weights also start to decay toward zero.

This behavior illustrates a fundamental structural feature of the MEWLS framework: 
the classification of data points as inliers or outliers is not imposed 
a priori, but emerges continuously along the solution path as a consequence 
of the global entropy-accuracy trade-off.

\begin{figure}
\begin{center}
\includegraphics[width=.5\textwidth]{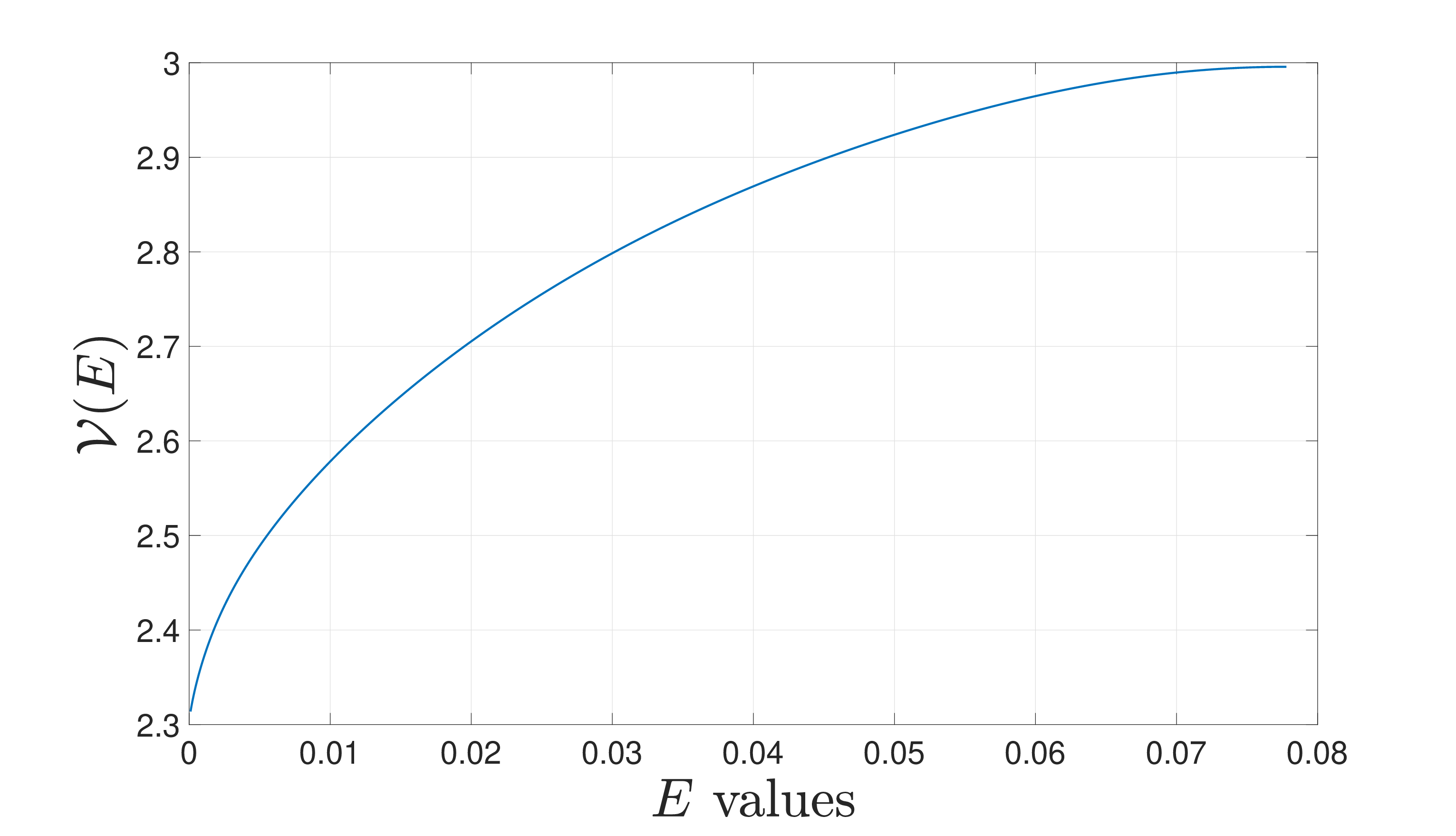} \hspace*{-0.5cm}
\includegraphics[width=.5\textwidth]{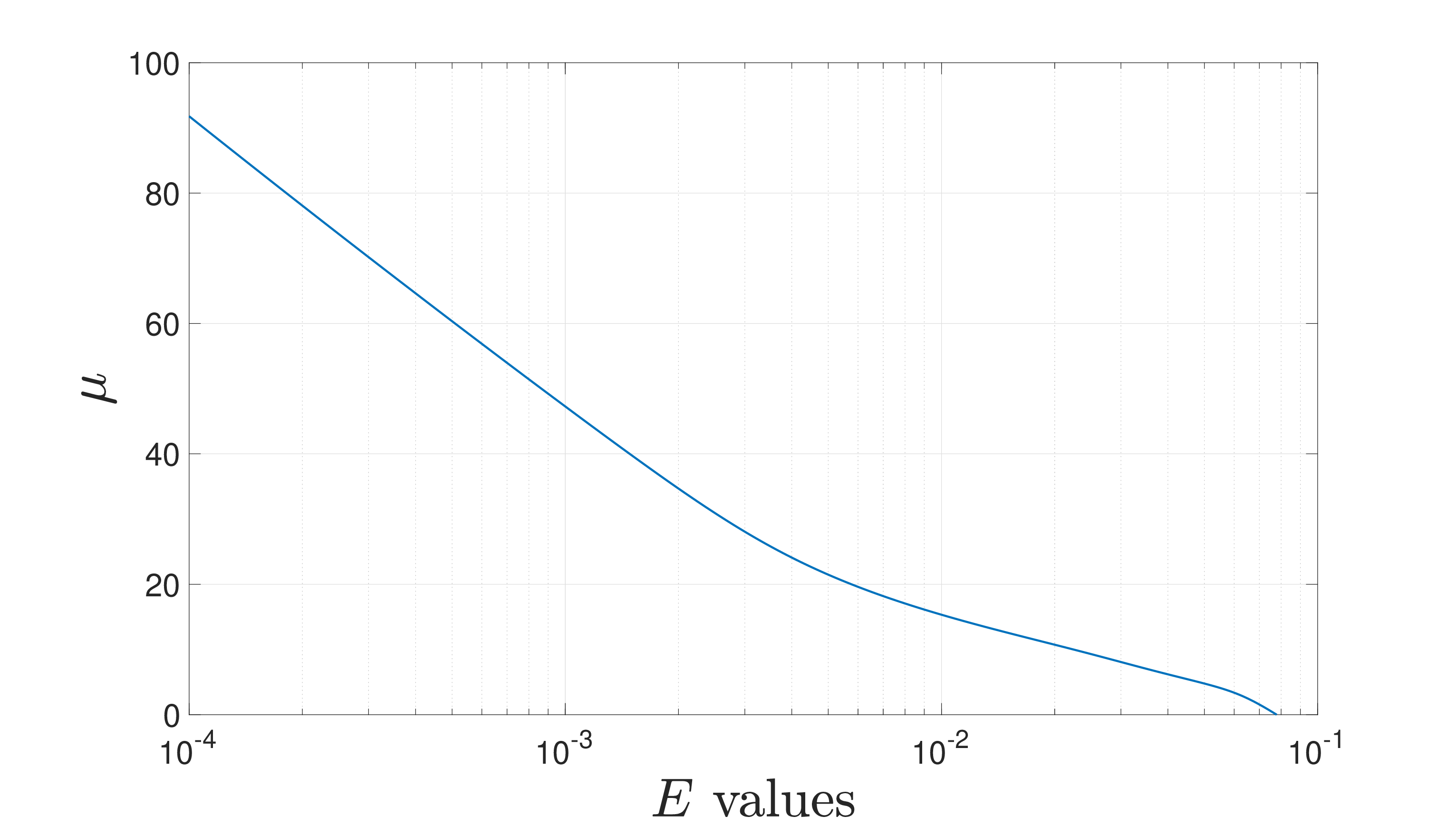} 
\caption{Left: value function $\mathcal V(E)$, exhibiting strict concavity 
as predicted by Proposition~\ref{prop:mu_monotone}. 
Right: Lagrange multiplier $\mu(E)$ plotted with a semilogarithmic scale 
on the horizontal axis, confirming positivity, strict monotonicity, 
and the asymptotic behavior described in Lemma~\ref{lemma:mu-log}.  
\label{fig3}}
\end{center}
\end{figure}

In the left panel of Figure~\ref{fig3} we plot the value function $\mathcal V(E)$. 
Its strict concavity is clearly visible, in agreement with 
Proposition~\ref{prop:mu_monotone}.

The right panel of Figure~\ref{fig3} shows the Lagrange multiplier $\mu(E)$ 
with a semilogarithmic scale on the horizontal axis. The plot confirms:
\begin{itemize}
\item strict positivity of $\mu(E)$ (Proposition~\ref{prop:mu_monotone});
\item strict monotonicity with respect to $E$ (Proposition~\ref{prop:mu_monotone});
\item the asymptotic behavior described in Lemma~\ref{lemma:mu-log}.
\end{itemize}

The left panel of Figure~\ref{fig4} reports the smallest eigenvalue of the matrix 
$\widehat S$ defined in \eqref{Shmatrix}. 
Its positivity is equivalent to the positive definiteness of the matrix 
$B(y)$ (defined in \eqref{Bmatrix}) on the subspace
$V := \mathrm{Range}(W^{1/2}A) \subseteq \mathbb{R}^m$, 
as established in Proposition~\ref{prop:B_positive_on_V}.

The numerical evidence confirms that this structural property 
is preserved along the entire solution branch, thus supporting 
the theoretical framework developed in the previous sections.

\begin{figure}
\begin{center}
\includegraphics[width=.5\textwidth]{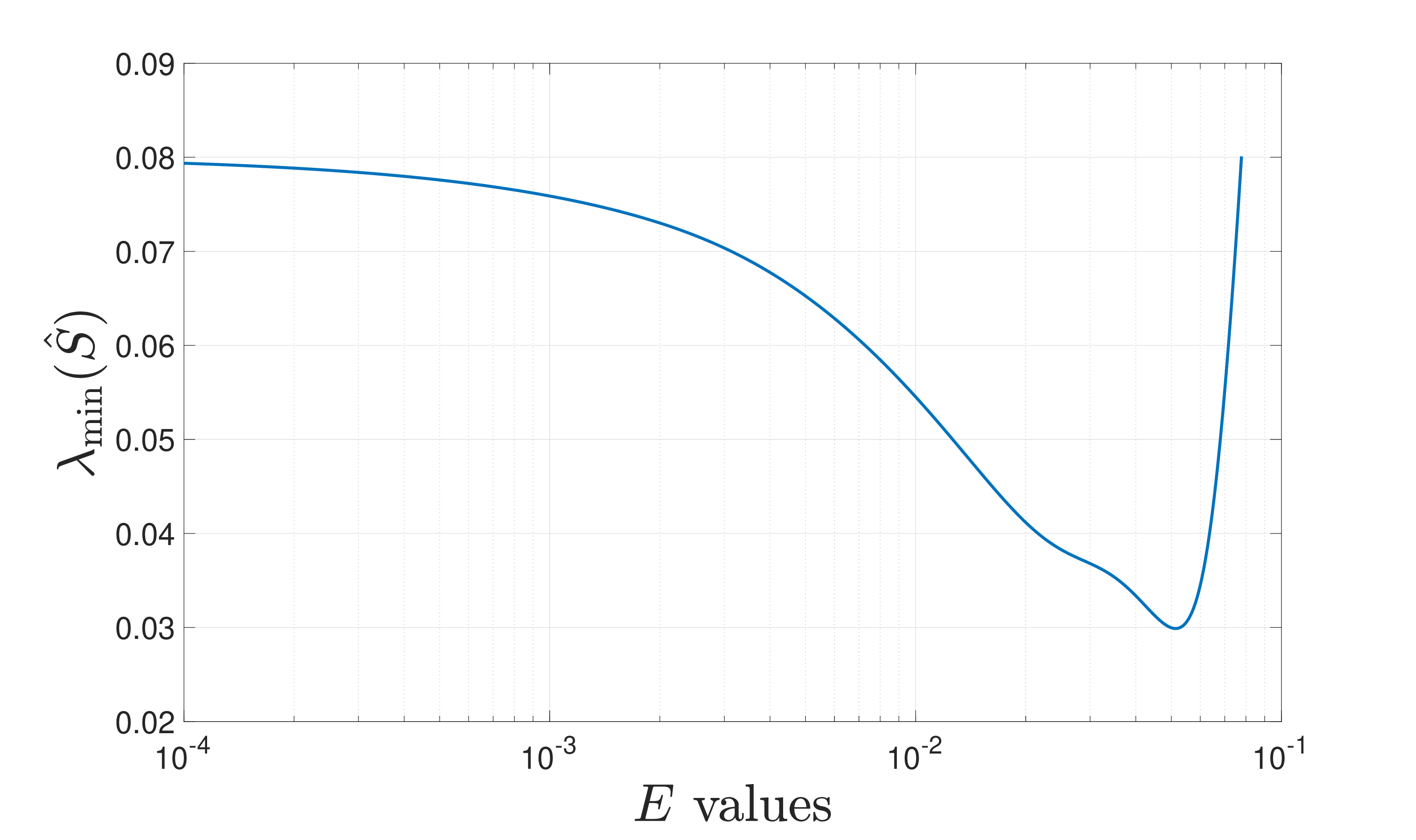} \hspace*{-0.5cm}
\includegraphics[width=.5\textwidth]{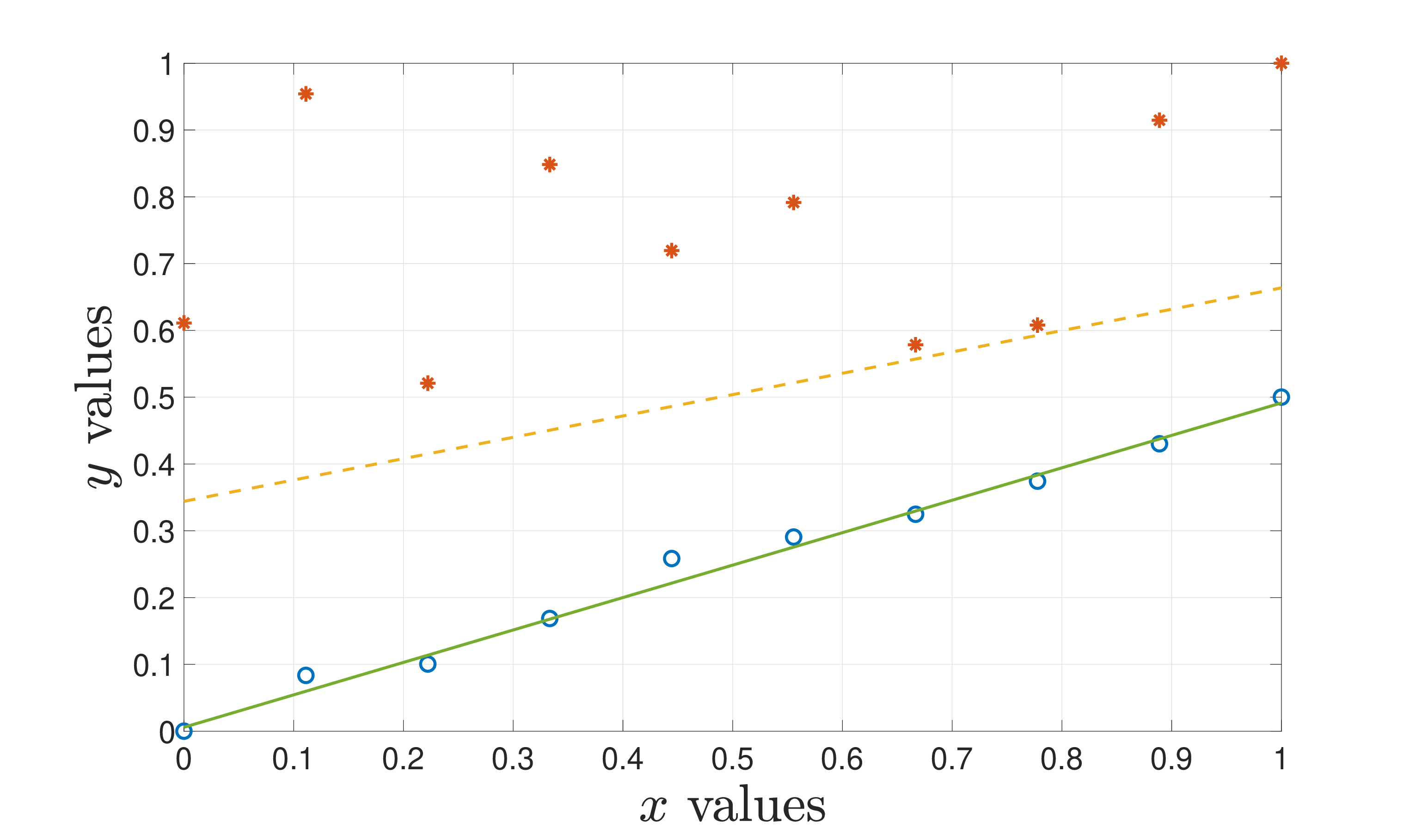} 
\caption{Left: smallest eigenvalue of the matrix $\widehat S$ along the solution branch, 
confirming the positivity property stated in 
Proposition~\ref{prop:B_positive_on_V}. 
Right: noisy-inlier experiment. 
Yellow dashed line: OLS regression in the presence of outliers. 
Green solid line: MEWLS regression obtained after reducing the MSE to the 
outlier-free level, showing effective suppression of the outliers.  
\label{fig4}}
\end{center}
\end{figure}

Finally, to mimic a more realistic scenario, where benign data do not exactly 
conform to the underlying model, we repeat the experiment by introducing 
Gaussian noise with zero mean and variance $\sigma^2=3\cdot 10^{-2}$ 
in the ordinates of the inliers.

In the absence of outliers, the ordinary least squares (OLS) solution 
with uniform weights produces the regression line
\[
y\approx 1.50\cdot 10^{-2}+4.71 \cdot 10^{-1} x,
\]
with corresponding uniform-weight mean squared error
\begin{equation}
\label{MSEuw1}
\mathrm{MSE}_{\mathrm{uw}} \simeq 9.41\cdot 10^{-4}.
\end{equation} 

In the presence of outliers, the OLS approximation yields the regression line
\[
y\approx 3.40\cdot 10^{-1}+3.14 \cdot 10^{-1} x,
\]
shown as the yellow dashed line in the right panel of Figure~\ref{fig4}.  

We then solve the MEWLS problem by reducing the MSE down to the value 
\eqref{MSEuw1}.\footnote{Practical strategies for selecting a suitable 
reduction factor are discussed in \cite{DeFaIaLoMaRu25}.} In this setting, the inlier data are no longer exactly interpolatory, 
but lie on the underlying model \eqref{model1} up to small Gaussian noise. 
Hence, the assumptions of Proposition~\ref{prop:weights_on_S_uniform} 
are only approximately satisfied. 

Nevertheless, the same separation mechanism is observed numerically: 
the weights associated with the outliers decay to values close to zero, 
while the weights corresponding to the inliers remain approximately balanced. 

The resulting regression line, depicted as the green solid line in the right panel 
of Figure~\ref{fig4}, has equation
\[
y\approx 1.46\cdot 10^{-2}+4.78 \cdot 10^{-1} x,
\]
which is close to the outlier-free scenario. 
This confirms that the MEWLS mechanism remains effective even when the 
interpolatory assumption is relaxed and only approximately satisfied.

As observed in the introduction and discussed in Section \ref{sec_formulation},  the MEWLS problem is, in general, nonconvex  and may admit multiple local maximizers of the entropy functional under 
the MSE constraint. 
In the present example, the existence of a sufficiently large subset 
$S$ of points  conforming to the model (inliers)  ensures that the initial 
uniform-weight configuration lies within the basin of attraction of the 
branch associated with $S$. As the MSE decreases, the solution path 
continuously evolves toward a configuration where the weights concentrate 
uniformly on $S$, while the complementary weights vanish.

However, if the number of outliers/inliers is progressively increased/decreased, this 
dominance condition may fail. In such a situation, the uniform-weight 
initialization may fall into the basin of attraction of a different 
local maximizer, and the resulting regression line converges toward 
a distinct local solution. 

Nevertheless, extensive numerical experience on datasets of different 
nature and arising from heterogeneous application contexts, together 
with comparisons against established robust regression techniques, 
suggests that the MEWLS approach exhibits a remarkably high breakdown 
point, thereby providing a significant degree of robustness in practice.

\subsection{Example 2}
In this example we illustrate two different breakdown events caused by the loss 
of invertibility of the Jacobian matrix $J(y(E))$ along the solution branch.

We first consider four data points with abscissae and ordinates given by
$$
\begin{array}{r|llll}
x & 0.3 & 0.3 & 0.7 & 0.7 \\
\hline  
y & 0.4 & 0.6 & 0.4 & 0.6
\end{array}
$$
These points (blue circles in the left panel of Figure~\ref{fig5}) are 
symmetrically distributed with respect to the horizontal line $y=1/2$, and lie at equal distance from it.

\begin{figure}
\begin{center}
\includegraphics[width=.5\textwidth]{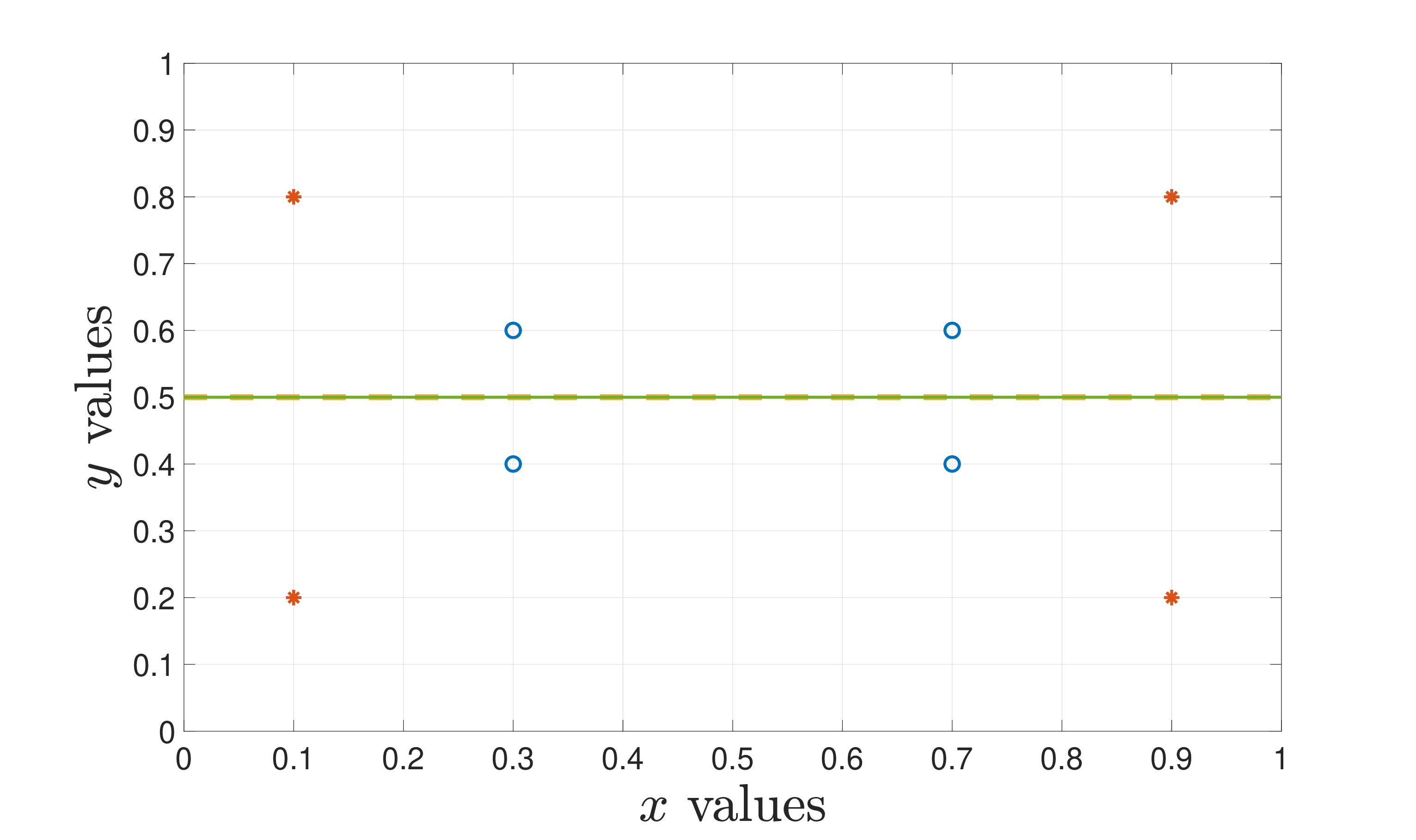} \hspace*{-0.5cm}
\includegraphics[width=.5\textwidth]{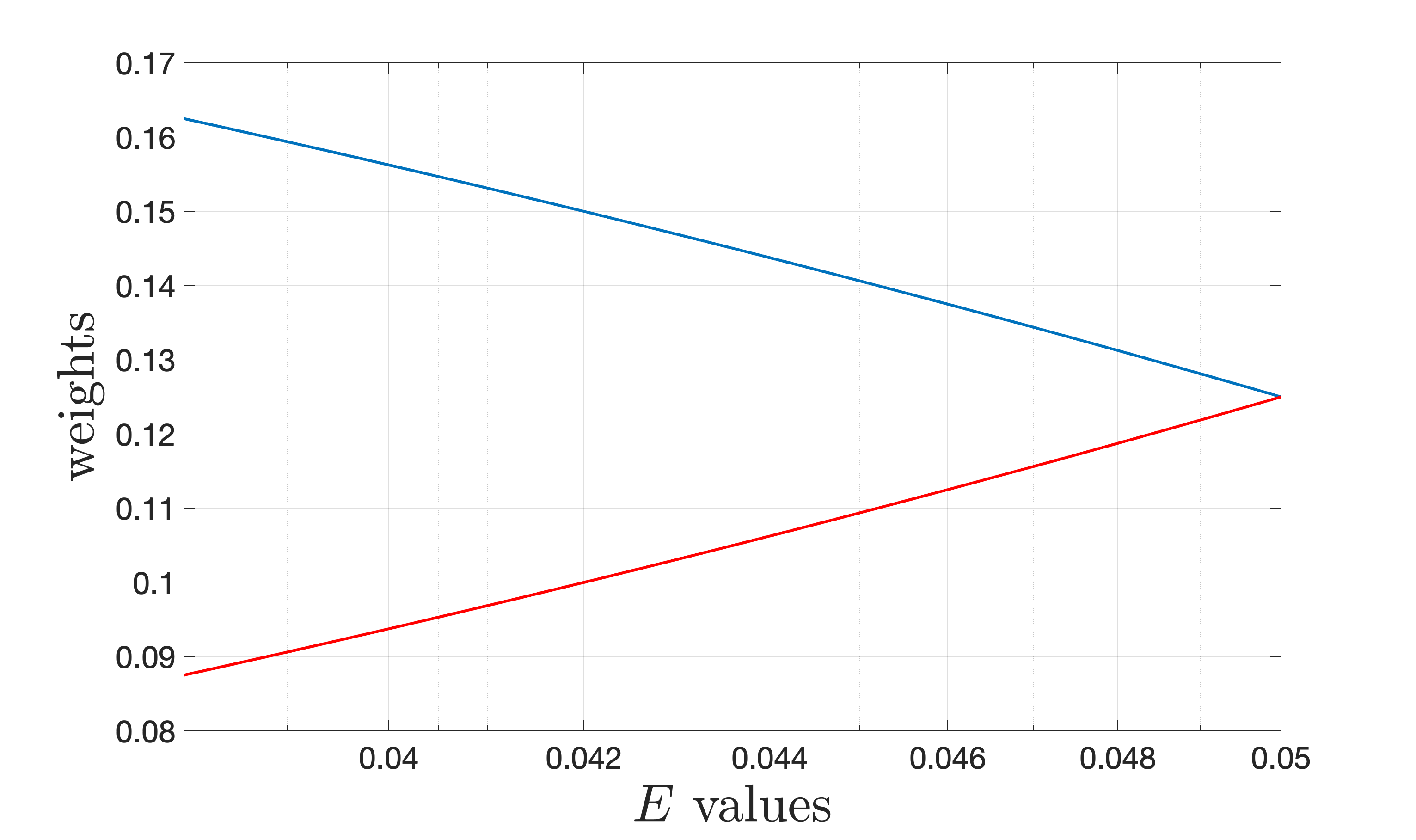} 
\caption{Left: initial uniform-weight regression line (dashed yellow) and MEWLS regression line (solid green) for the symmetric dataset.  Because of symmetry with respect to $y=1/2$, the regression line remains unchanged along the branch as $E$ decreases.  Right: evolution of the optimal weights along the solution branch. 
The weights associated with the points farthest from $y=1/2$ (red asterisks) decrease monotonically (red solid line), whereas those corresponding to the closer points (blue circles) increase symmetrically (blue solid line), illustrating the entropy redistribution process.
 \label{fig5}}
\end{center}
\end{figure}

By symmetry, the ordinary least squares regression line with uniform 
weights coincides with $y=1/2$. In this configuration, the vector 
of squared residuals $r^2$ is constant.  According to Theorem \ref{theo1} and Proposition~\ref{prop:J-invertible}, this implies that
$\det J(y(E_{\mathrm{uw}})) = 0$,
so that the Jacobian matrix is singular at the uniform-weight solution. Consequently,  the existence and uniqueness of the branch is not guaranteed even locally.

We now enrich the dataset by adding four further points 
(red asterisks in the left panel of Figure~\ref{fig5}) defined as 
$$
\begin{array}{r|llll}
x & 0.1 & 0.1 & 0.9 & 0.9 \\
\hline  
y & 0.2 & 0.8 & 0.2 & 0.8
\end{array}
$$
The enlarged dataset still exhibits perfect symmetry with respect to 
the line $y=1/2$, but the squared residual vector $r^2$ is no longer constant.

Because of symmetry, the regression line does not change as $E$ decreases: 
the component $x(E)$ of the solution branch remains constant. 
The initial regression line (dashed yellow) and the final configuration 
(solid green) are therefore indistinguishable in the left panel of 
Figure~\ref{fig5}.

Although the regression line remains fixed, the weights exhibit a 
nontrivial evolution (right panel of Figure~\ref{fig5}).  The four weights associated with the points farthest from the line  $y=1/2$ decrease monotonically and, by symmetry, follow the same law. 
Conversely, the four weights corresponding to the points closer to the 
line increase in a symmetric and complementary fashion as $E$ decreases. Thus, even though the geometric configuration of the regression line remains unchanged, the entropy redistribution mechanism remains active.
\begin{figure}
\begin{center}
\includegraphics[width=.5\textwidth]{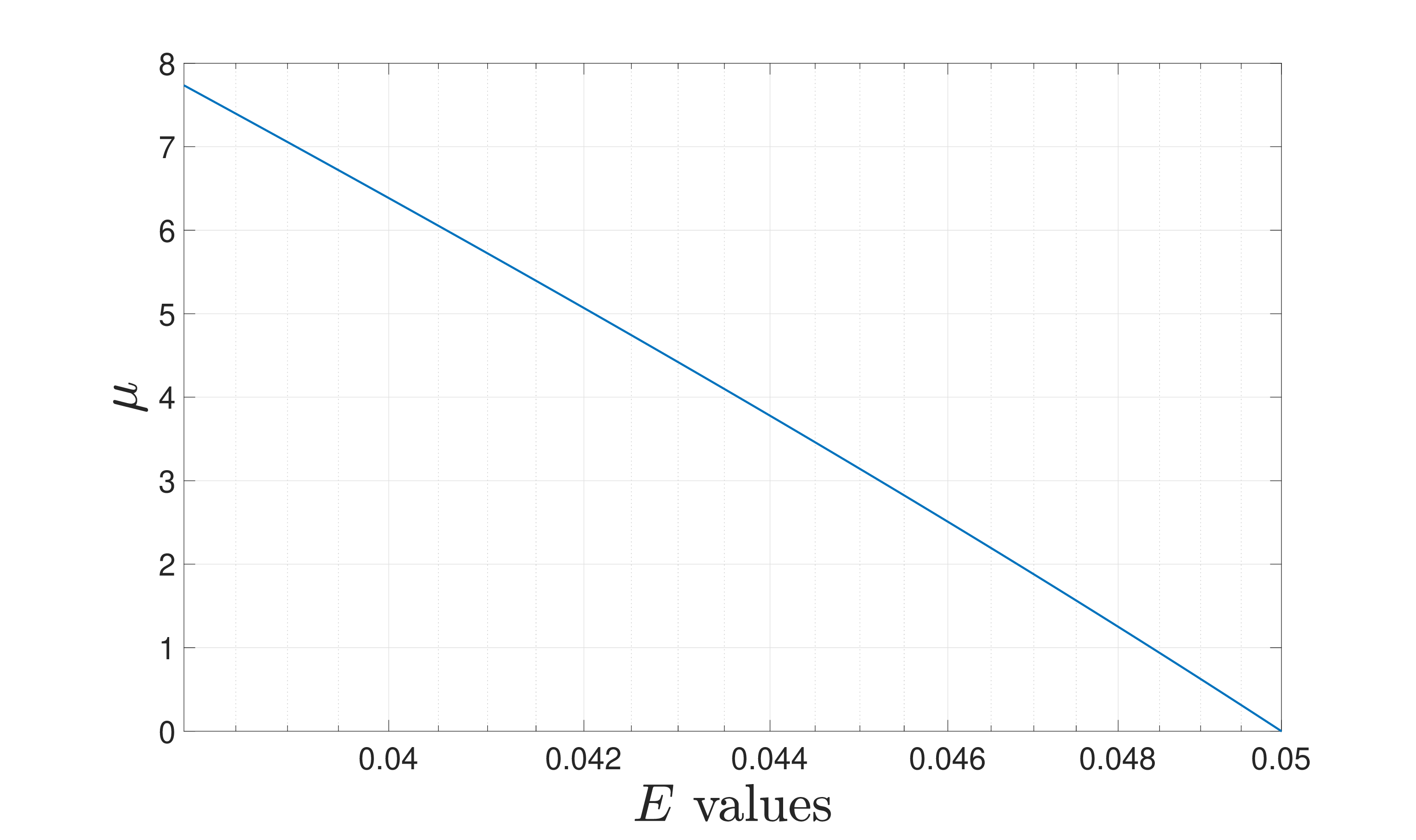} \hspace*{-0.5cm}
\includegraphics[width=.5\textwidth]{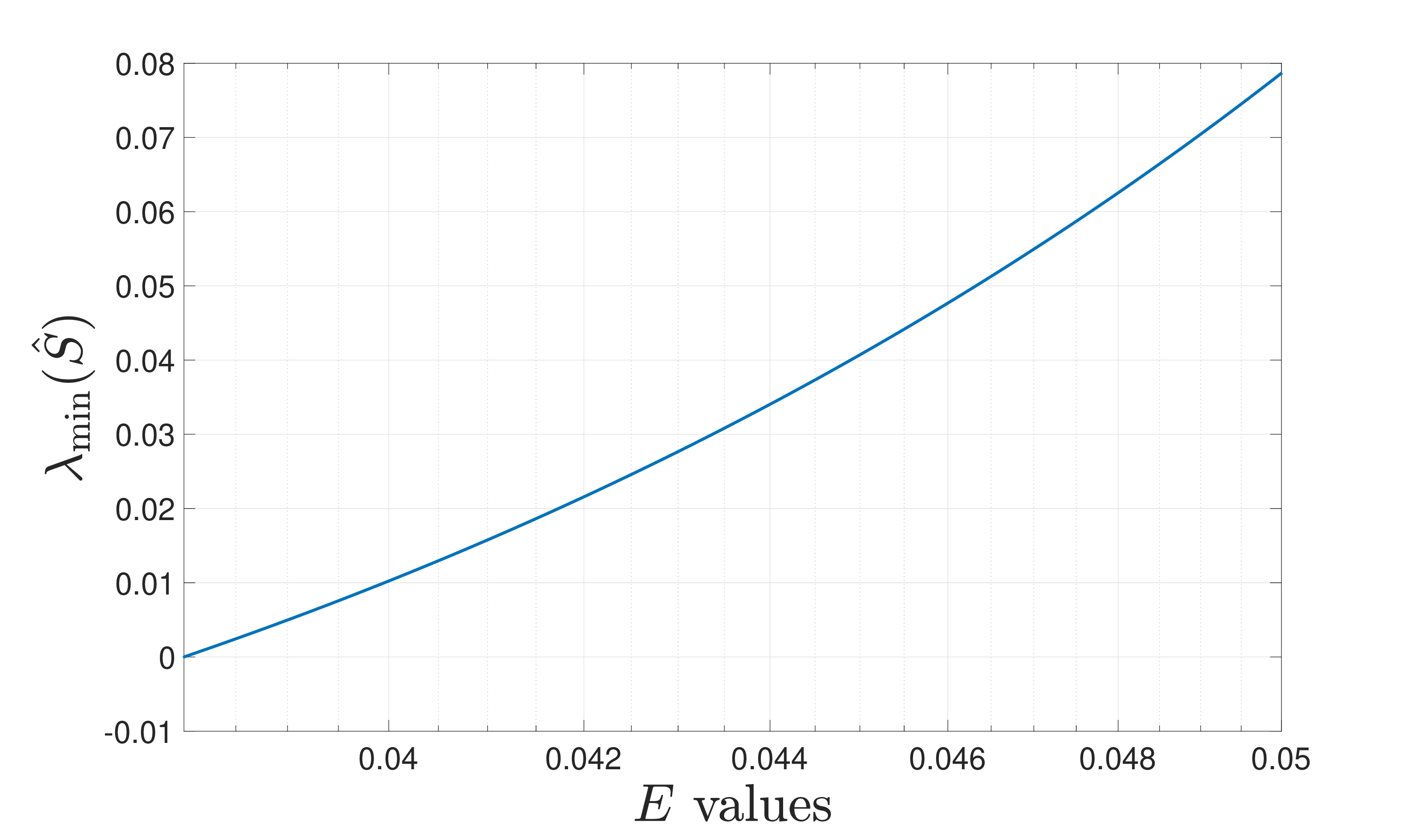} 
\caption{Left: monotone behavior of the Lagrange multiplier $\mu(E)$ along the solution branch. 
Right: smallest eigenvalue of the matrix $\widehat S$ as a function of $E$. 
At $E \approx 3.80 \cdot 10^{-2}$, the smallest eigenvalue vanishes, so that the Jacobian $J(y(E))$ loses invertibility. This loss of regularity prevents further continuation of the solution branch.
 \label{fig6}}
\end{center}
\end{figure}

The monotonic behavior of the multiplier $\mu(E)$ is shown in the left 
panel of Figure~\ref{fig6}. The right panel of the same figure reports 
the smallest eigenvalue of the matrix $\widehat S$. We observe that for
$E \approx 3.80 \cdot 10^{-2}$, the smallest eigenvalue vanishes, and the matrix $\widehat S$ becomes singular. 
By Proposition~\ref{prop:J-invertible}, this implies that the Jacobian 
matrix $J(y(E))$ becomes singular at this value of $E$. Consequently, the 
solution branch cannot be further prolonged.

Although the present example is deliberately constructed and represents 
an exceptional symmetric configuration, it highlights a structural 
mechanism of breakdown: when the data cloud becomes symmetrically 
distributed with respect to the regression model, the information 
content carried by the residual structure degenerates.

In practical situations, a similar phenomenon may occur when 
$\det J(y(E))$ becomes very small, even if not exactly zero. 
This signals that the entropy-driven selection process is starting 
to discard data points that still carry relevant geometric information. 
From a computational perspective, this can provide a natural stopping 
criterion for the continuation procedure.

%%%%%%%%%%%%%%%%%%%%%%%%%%%%%
\section{Conclusions}
\label{sec_conclusions}
We have studied the structure of stationary points of an entropy-constrained weighted least squares problem for overdetermined linear systems. 
The nonconvexity induced by the residual constraint prevents a global convexity-based analysis and requires a local investigation of solution branches.

We have established existence and uniqueness of a smooth branch of stationary solutions emanating from the ordinary least squares configuration. 
By differentiating the first-order optimality system with respect to the MSE parameter, we obtained a differential system describing the evolution of the branch and providing a continuation framework up to possible singular points characterized by loss of Jacobian invertibility.

The asymptotic analysis as the prescribed MSE tends to zero shows that the entropy-maximizing solution may concentrate on a largest consistent subset of observations. 
This provides a theoretical explanation of the selective behavior of the method in the presence of anomalous data.

Further work will focus on the efficient numerical solution of the differential system arising from the continuation framework, exploiting the sparsity and structured nature of the associated Jacobian matrix and investigating suitable integration strategies.

\section*{Acknowledgements}
The first author is a member of the ``Gruppo Nazionale per il Calcolo Scientifico--Istituto Nazionale di Alta Matematica (GNCS--INdAM)''.

%%%%%%%%%%%%%%%%%%%%%%%%%%%%%%%

\end{document}